\documentclass[12pt]{amsart}
\usepackage{MnSymbol}
\usepackage{mathrsfs}
\usepackage{amsmath,amsthm}

\usepackage[all]{xy}

\numberwithin{equation}{section}

\newtheorem{Theorem}{Theorem}[section]
\newtheorem{Corollary}[Theorem]{Corollary}
\newtheorem{Lemma}[Theorem]{Lemma}
\newtheorem{Proposition}[Theorem]{Proposition}
\theoremstyle{definition}
\newtheorem{Remark}[Theorem]{Remark}
\newtheorem{Example}[Theorem]{Example}
\newtheorem{Definition}[Theorem]{Definition}

\numberwithin{equation}{section}

\DeclareMathAlphabet\mathbb{U}{msb}{m}{n}

\def\Mod{{{\sf Mod}}}
\def\C{{\mathcal C}}
\def\D{{\mathcal D}}
\def\F{{\mathcal F}}
\def\G{{\mathcal G}}
\def\lim{{\sf lim}}

\def\N{{\sf N\hspace{1pt}}}
\def\B{{\sf B\hspace{1pt}}}
\def\mono{\rightarrowtail}
\def\epi{\twoheadrightarrow}
\def\f{{\bf f}}
\def\r{{\bf r}}
\def\a{{\bf a}}
\def\b{{\bf b}}
\def\Z{{\mathbb Z}}
\def\O{{\mathcal O}}

\begin{document}

\title{Higher limits, homology theories and fr-codes}
\author{Sergei O. Ivanov}
\address{Chebyshev Laboratory, St. Petersburg State University, 14th Line, 29b,
Saint Petersburg, 199178 Russia} \email{ivanov.s.o.1986@gmail.com}

\author{Roman Mikhailov}
\address{Chebyshev Laboratory, St. Petersburg State University, 14th Line, 29b,
Saint Petersburg, 199178 Russia and St. Petersburg Department of
Steklov Mathematical Institute} \email{rmikhailov@mail.ru}

\maketitle

\begin{abstract}
This text is based on lectures given by authors in summer 2015. It
contains an introduction to the theory of limits over the category
of presentations, with examples of different well-known functors
like homology or derived functors of non-additive functors in a
form of derived limits. The theory of so-called ${\bf fr}$-codes
also is developed. This is a method how different functors from
the category of groups to the category of abelian groups, such as
group homology, tensor products of abelianization, can be coded as
sentences in the alphabet with two symbols ${\bf f}$ and ${\bf
r}$.
\end{abstract}

\vspace{1cm}

These are notes based on lectures given by authors in 2015. The
main idea of the theory presented in this text is the following.
Given a group (or analogously, associative algebra, Lie algebra
over a ring, abelian group etc) $G$, we consider the category
${\sf Pres}(G)$. The objects of ${\sf Pres}(G)$ are surjective
homomorphisms $\pi:F\epi G$ where $F$ is a free group and
morphisms $f:(\pi_1:F_1\epi G)\to (\pi_2:F_2\epi G)$ are
homomorphisms $f:F_1\to F_2$ such that $\pi_1=\pi_2f$. The
category ${\sf Pres}(G)$ has coproducts and hence is contractible.
However, cohomology group of ${\sf Pres}(G)$ with coefficients in
functors define invariants of $G$. This provides an approach to
functors in the category of groups (or associative algebras, Lie
algebras, or in general, in the categories with enough
projectives). Here is an illustrative example from
\cite{Ivanov-Mikhailov}. Let $G$ be a 2-torsion-free group. For a
given $\pi: F\epi G$, denote $R:=\ker(\pi)$. Then there is a
natural isomorphism
$$
H^1(B{\sf Pres}(G), \frac{[R,R]}{[[R,R],F]})=\lim^1
\frac{[R,R]}{[[R,R],F]}=H_3(G;\mathbb Z/2),
$$
where $B{\sf Pres}(G)$ is the classifying space of the category
${\sf Pres}(G)$. That is, one can obtain the third homology of $G$
with $\mathbb Z/2$-coefficients as the first derived limit of the
presentation $\frac{[R,R]}{[[R,R],F]}$ over the category ${\sf
Pres}(G)$. A lot of different functors like group homology, cyclic
and Hochschild homology of algebras, homology of Lie algebras,
derived functors in the sense of Dold-Puppe can be obtained in
this way. We refer to \cite{Quillen2}, \cite{Ivanov-Mikhailov},
\cite{EM}, \cite{MP1} for discussion of different aspects of this
approach.

One of the subjects of these notes is the theory of so-called
${\bf fr}$-codes. We consider an alphabet with two symbols ${\bf
f, r}$ and form {\it sentences} with words in this alphabet, for
example, $({\bf rfr}\ {\bf frr}\ {\bf ffff})$ or $({\bf rr}\ {\bf
frf})$. Now, looking at the object $\pi: F\epi G$, we construct an
ideal in $\mathbb Z[F]$ given by the sum of words in a sentence
where we replace $\bf f$ by augmentation ideal of $F$ and $\bf r$
by the ideal $(\ker(\pi)-1)\mathbb Z[F]$. The {\it translation} of
a sentence (or ${\bf fr}$-code) on the language of functors in the
category of groups is the corresponding to that ideal
$\lim^1$-functor. For example, the above sentences give natural
functors:
\begin{align*}
& ({\bf rfr}\ {\bf frr}\ {\bf ffff})\rightsquigarrow\ \lim^1({\bf
rfr}+{\bf frr}+{\bf ffff})={\sf Tor}(G_{ab}\otimes G_{ab},
G_{ab}),\\
& ({\bf rr}\ {\bf frf})\ \rightsquigarrow \lim^1({\bf rr}+{\bf
frf})=H_3(G).
\end{align*}
Here derived limits are taken over ${\sf Pres}(G)$. We give a
table of different ${\bf fr}$-codes at the end of these notes. The
authors thank F. Pavutnitsky for discussions related to the
subject of these notes.

\section{Background}

\subsection{Tarski-Grothendieck set theory}

Through the paper we use the advanced version of axiomatic set
theory called Tarski-Grothendieck set theory.  It is an extension
of the standard Zermelo-Fraenkel set theory with the axiom of
choice (ZFC) by the Tarski's axiom of universes (U). Here we
recall what is it (for details see \cite{Tarski},
\cite{Boubaki_univers}, \cite{Shulman},\cite{Braids}).

A {\bf Grothendieck universe} is a set $U$ with the following properties:
\begin{enumerate}
\item if $x\in U$ and $y\in x,$ then $y\in U;$
\item $\emptyset \in U;$
\item if $x\in U,$ then $\mathcal P(x)\in U$ (where $\mathcal P(x)$ is the power-set);
\item if $\{x_i\}_{i\in I}$ is a family such that $x_i\in U$ and $I\in U,$ then $\bigcup_{i\in I}x_i\in U.$
\end{enumerate}
It is similarly easy to prove that any Grothendieck universe $U$ contains:
\begin{itemize}
\item all singletons of each of its elements;
\item all couples of its elements;
\item all products of all families of elements of $U$ indexed by an element of $U$;
\item all disjoint unions of all families of elements of $U$ indexed by an element of $U$;
\item all intersections of all families of elements of $U$ indexed by an element of $U$;
\item all functions between any two elements of $U$;
\item all subsets of $U$ whose cardinal is an element of $U$.
\end{itemize}

The Tarski's axiom (U) is the following.

\begin{itemize}
\item[(U)] For each set $x$ there is a Grothendieck universe $U$ such that $x\in U.$
\end{itemize}

A Grothendieck universe is a  set closed under all set-theoretical constructions.
It follows that, if $\mathbb N\in U,$ then $\mathbb Z, \mathbb Q, \mathbb R, \mathbb C, \mathbb R^n, \mathcal P(\mathbb R^n), (\mathbb R^n)^{\mathbb R^m}$ and all their subsets are in $U.$ Roughly speaking, if $\mathbb N\in U,$ then all classical mathematics is in $U.$ More precisely, if $\mathbb N\in U$ then $U$ is a model of ZFC set theory. Moreover, we fix a couple of universes $U,U'$ such that $U\in U'$, and start to call elements of $U$ by sets and elements of $U'$ by classes we get a model for NBG class theory. So considering of universes allows to avoid considering classes. Moreover, it is more flexible because we can consider a sequence of universes $U_0\in U_1\in U_2\in \dots$ and call elements of $U_n$ classes of $n$th level and work with them as with sets because they are sets in the Tarski-Grothendieck set theory.

A set $x$ is called $U$-small for some Grothedieck universe $U$ if $x\in U.$ A topological space (group, ring, module,\dots) is called $U$-small if its underlying set is $U$-small. If $U$ is fixed, we call them just small.  A category is called $U$-small if the set of objects and the set of morphisms are $U$-small. In order to avoid using classes we will assume that {\bf all categories are small}. It follows that any category is $U$-small for some Grothendieck universe $U.$ We denote by ${\sf Sets}^U,$ ($ {\sf Top}^U, {\sf Gr}^U, {\sf Rings}^U, {\sf Mod}(R)^U,$ ${\sf Cat}^U$ \dots) categories of $U$-small sets (topological spaces, groups, rings, $R$-modules, categories \dots). If $U\in U',$ where $U'$ is another Grothendieck universe, then all these categories are $U'$-small. Further we will avoid the superscript $U$ and assume that $U$ is a  Grothendieck universe such that $\mathbb N \in U$.

\subsection{Simplicial sets}
In this subsection we recall some basic information about simplicial sets. For more information and proofs see \cite{May}, \cite{Gabriel_Zisman},  \cite{Goerss_Jardine}.

The  {\bf simplicial indexing category} is the category $\Delta$ whose objects are posets
$$[n]=\{0<1<\dots<n\}$$ for $n\geq 0$ and morphisms are order-preserving maps. For $n\geq 1$ and $0\leq i\leq n$ we define $i$th {\bf coface Map} $d^i:[n-1]\to [n]$ as the unique injective order-preserving map whose image does not contain $i.$
For $n\geq 0$ and $0\leq i\leq n$ we define $i$th {\bf codegeneracy Map} $s^i:[n+1]\to [n]$ as the unique sujective order-preserving map such that $i$ has two-element preimage. It is easy to check that they satisfy the following relations:
 (i) $d^jd^i=d^id^{j-1}$ if $i<j$; (ii) $s^js^i=s^{i-1}s^j$ if $i>j;$ (iii) $s^jd^i=d^is^{j-1}$ if $i<j;$  (iv) $s^{i}d^i={\sf id}=s^{i+1}d^i;$ (v) $s^jd^i=d^{i-1}s^j$ if $i>j+1.$
It follows that any order-preserving map $\alpha:[m]\to [n]$ can be written is one and only one way as
$$\alpha=d^{j_l}\dots d^{j_1}s^{i_k}\dots s^{i_1},$$
where $m>i_1>\dots >i_k\geq 0$ and $0 \leq j_1<\dots <j_l\leq n.$ Using this one can prove that $\Delta$ is the path category of the quiver
$$
\xymatrix{
 \dots &
[2]\ar@<-6pt>[r]|{s^0}\ar@<-12pt>[r]|{s^1} &
[1]\ar@<-6pt>[r]|{s^0} \ar@<-12pt>[l]|{d^0}\ar@<-6pt>[l]|{d^1} \ar[l]|{d^2} &
[0]\ar@<-6pt>[l]|{d^0} \ar[l]|{d^1}}
$$
with relations (i)-(v).

Let $\C$ be a category.  A {\bf simplicial object} of $\C$  is a functor $X:\Delta^{\rm op}\to \C.$ We set $X_n=X([n])$, $d_i=X(d^i)$ and $s_i=X(s^i).$ A morphism of simplicial objects is a natural transformation. The category of simplicial objects is denoted by ${\sf s}\C.$ Using that $\Delta$ is a
path category with relations, we can define a simplicial set as a sequence of objects $X_0,X_1,\dots$ together with {\bf face morphisms} $d_i:X_n\to X_{n-1}$ and {\bf degeneracy morphisms} $s_i:X_n\to X_{n+1}$ for $0\leq i\leq n$ satisfying identities:
\begin{itemize}
\item[(s1)] $d_id_j=d_{j-1}d_i$ if $i<j$;
\item[(s2)] $s_is_j=s_js_{i-1}$ if $i>j;$
\item[(s3)] $d_is_j=s_{j-1}d_i$ if $i<j;$
\item[(s4)] $d_is_{i}={\sf id}=d_is_{i+1};$
\item[(s5)] $d_is_j=s_jd_{i-1}$ if $i>j+1.$
\end{itemize}

The most important case of $\C$ is ${\sf Sets}.$ Simplicial objects of ${\sf Sets}$ are called {\bf simplicial sets}. If $X$ is a simplicial set, elements of $X_n$ are called {\bf $n$-simplexes} of $X.$ $n$-simplex $x$ is {\bf degenerate} if $x=s_i(x')$ for some $i$ and $x'\in X_{n-1}.$

   The simplicial set  $\Delta[m]=\Delta(-,[m])$ is called  simplicial $m$-simplex. It defines a functor
$$\Delta[-]:\Delta \longrightarrow {\sf  sSets}$$
$$[m]\longmapsto \Delta[m] $$
which is the Yoneda embedding for the category $\Delta.$ For simplicial sets $X,Y$ we denote by $X\times Y$ their dimension-wise product (which is the product in the categorical sense) and  by ${\sf Map}(X,Y)$ the {\bf function complex} (or {\bf internal hom}) which is the simplicial set defined by
$${\sf Map}(X,Y)={\sf sSets}(\Delta[-]\times X ,Y).$$ It is easy to check that there is a natural isomorphism  ${\sf sSets}(X\times Y,Z)\cong {\sf sSets}(X,{\sf Map}(Y,Z)).$

A (simplicial) {\bf homotopy} between two simplicial maps $f,g:X\to Y$ is an element $H\in {\sf Map}(X,Y)_1$ such that $d_1(H)=f$ and $d_0(H)=g.$ In other words, it is a simplicial map $H:\Delta[1]\times X \to Y$ such that $H(\Delta[d^1]\times {\sf id})=f$ and $H(\Delta[d^0]\times {\sf id})=g.$ Here we identify $\Delta[0]\times X$ with $X.$

Let $\Delta_{\sf Top}[n]$ denotes the standard $n$-simplex in $\mathbb R^{n+1}.$ Then any order-preserving $\alpha:[m]\to [n]$ induces a linear map $\mathbb R^{m+1}\to \mathbb R^{n+1}$ that sends $e_{i+1}$ to $e_{\alpha(i)+1}.$ This map can be restricted to the map $\Delta_{\sf Top}[\alpha]:\Delta_{\sf Top}[m]\to \Delta_{\sf Top}[n].$ Hence, we defined a functor to the category of topological spaces
$$\Delta_{\sf Top}[-]:\Delta\longrightarrow {\sf Top}.$$

There is a pair of adjoint functors between categories ${\sf Top}$ and ${\sf sSets}$ that make these categories very close to each other. First of these functor is the {\bf singular simplicial set} functor
$$  {\sf Sing}:{\sf Top} \longrightarrow {\sf sSets}$$
given by
$${\sf Sing}(X)={\sf Top}\left(\Delta_{\sf Top}[-],X\right).$$
The second is the geometric realisation functor
$$|-|:{\sf sSets}\longrightarrow {\sf Top}.$$
Its constructive definition is little bit more complicated but it can be reconstructed from the equality $|\Delta[-]|=\Delta_{\sf Top}[-]$ and the fact that it commutes with all colimits because any simplicial set is a colimit of simplcial simplexes $\Delta[m]$. Moreover, geometric realisation functor is left adjoint to ${\sf Sing}$ i.e. for a simplicial set $X$ and a topological space $T$ there is a natural isomorphism
$${\sf Top}(|X|,T)\cong {\sf sSets}(X,{\sf Sing}(T)).$$

An important property of geometric realisations is it commutes with some products. For example, if $X$ and $Y$ are simplicial sets and one of them is finite (i.e. there are only finite number of non-degenerate simplexes), then there is a natural isomorphism of topological spaces
\begin{equation}\label{eq_prod}
|X\times Y|\cong |X|\times |Y|.
\end{equation}
In particular, any simplicial homotopy $H:\Delta[1]\times X\to Y$ defines a homotopy $|H|:I\times |X|\to |Y|,$ where $I=\Delta_{\sf Top}[1]$ is the closed interval.

Let $X$ be a simplicial set. Consider the complex ${\sf C}_\bullet(X)$ whose $n$th term ${\sf C}_n(X)$ is the free abelian group generated by $X_n$ and $d=\sum_i (-1)^id_i.$ Then homology and cohomology of $X$ with coefficients in an abelian group $A$ are defined as follows
$$H_n(X,A)=H_n({\sf C}_\bullet(X)\otimes A),\hspace{1cm} H^n(X,A)=H^n({\sf Hom}({\sf C}_\bullet(X),A)).$$
It is well known that
\begin{equation}\label{eq_homology}
H_n(X,A)\cong H_n(|X|,A), \hspace{1cm} H^n(X,A)\cong H^n(|X|,A).
\end{equation}

A simplicial set $X$ is {\bf pointed} if all sets $X_n$ are pointed and face and degeneracy maps preserve base points. In other words, a pointed simplicial set is a simplicial object in the category of pointed sets. The first item of the following proposition can be found in \cite[Prop. 2.17]{Braids}. The second item can be proved similarly.

\begin{Proposition}\label{Prop_Contractible_criterion}{\bf (Contractible criterion)} Let $X$ be a pointed simplicial set.
\begin{enumerate}
\item If there exist point preserving maps $\tilde s_{-1}:X_n\to X_{n+1}$ satisfying simplicial identities \hbox{(s2)-(s5)}, then $|X|$ is contractible.

\item If there exist point preserving maps $\tilde s_{n+1}:X_n\to X_{n+1}$ satisfying  simplicial identities (s2)-(s5), then $|X|$ is contractible.
\end{enumerate}
\end{Proposition}

\section{Classifying space of a category}

\subsection{Nerve of a category}

We consider posets as categories is a usual way and order-preserving maps as functors. Thus simplicial indexing category is a full subcategory of the category of small categories $\Delta\subset {\sf Cat}.$ The {\bf nerve} of a category $\C$ is a simplicial set $\N \C:\Delta^{\rm op} \to {\sf Sets}$ given by $$\N\C={\sf Cat}(-,\C)\mid_{\Delta^{\rm op}}.$$
For example $\Delta[n]=\N([n]).$
It is easy to see that there is a natural bijection between the set $(\N\C)_n={\sf Cat}([n],\C)$
and the set of all sequences of morphisms $(\alpha_1,\alpha_{2},\dots,\alpha_n)$ such that ${\sf codom}(\alpha_{i+1})={\sf dom}(\alpha_{i})$ in the category $\C.$
$$\bullet \xleftarrow{\ \alpha_1\ } \bullet \xleftarrow{\ \alpha_2\ } \bullet \xleftarrow{\ \alpha_3\ } \dots \xleftarrow{\ \alpha_n\ } \bullet $$
 Then face maps act as follows $d_i(\alpha_n,\dots,\alpha_1)=(\alpha_n,\dots,\alpha_{i+1}\alpha_i,\dots, \alpha_1)$ for $1\leq i< n ,$ $d_0(\alpha_n,\dots,\alpha_1)=(\alpha_n,\alpha_{n-1},\dots,\alpha_2)$ and $d_n(\alpha_n,\alpha_{n-1},\dots,\alpha_1)=(\alpha_{n-1},\dots,\alpha_1).$ Degeneracies act as follows $s_i(\alpha_n,\dots,\alpha_1)=(\alpha_n,\dots,\alpha_{i+1}, {\sf id}_{c_i},\alpha_i, \dots, \alpha_1)),$ where $c_i={\sf codom}(\alpha_i).$
 Roughly speaking,  $(\N\C)_n$ consists of commutative  $n$-simplexes in $\C.$
The nerve $\N\C$ is functorial by $\C$ i.e. it defines a functor  $\N: {\sf Cat}\to {\sf sSets}.$

If $\C,\D$ are categories, we denote by $\D^\C$ the category of functors whose objects are functors and morphisms are natural transformations. It is easy to see and well known that there is a natural isomorphism  ${\sf Cat}(\C\times \D,{\mathcal E})\cong {\sf Cat}(\C,{\mathcal E}^\D).$
\begin{Proposition} Let $\C$ and $\D$ be  categories. Then $\N :{\sf Cat}\to {\sf sSets}$ is a fully faithful functor i.e. it induces a bijection
$${\sf Cat}(\C,\D)\cong {\sf sSets}(\N \C,\N \D).$$
Moreover, there are natural isomorphisms $$\N(\C\times \D)\cong \N\C \times \N \D, \hspace{1cm} \N(\D^\C)\cong {\sf Map}(\N\C,\N\D).$$
\end{Proposition}
\begin{proof}
Let $f:\N \C\to \N\D$ be a simplicial map. We need to prove that there is a unique functor $\F:\C\to \D$ such that $f=\N\F.$ If $f=\N\F,$ then $\F(c)=f_0(c)$ and $\F(\alpha)=f_1(\alpha)$ for all objects $c$ and morphisms $\alpha.$ Hence, if it exists, it is unique. Prove that it exists. First we prove by induction on $n$ that $f_n(\alpha_n,\dots,\alpha_1)=(f_1(\alpha_n),\dots,f_1(\alpha_1))$ for $n\geq 1.$ The basis is obvious. Let $f_n(\alpha_n,\dots,\alpha_1)=(\beta_n,\dots,\beta_1).$ Then $(\beta_n,\dots,\beta_2)=d_0(\beta_n,\dots,\beta_1)= d_0f_n(\alpha_n,\dots,\alpha_1)=f_{n-1}(\alpha_n,\dots,\alpha_2)=(f_1(\alpha_n),\dots,f_1(\alpha_2))$  and $(\beta_{n-1},\dots,\beta_{1})=d_0(\beta_{n-1},\dots,\beta_{1})= d_0f_n(\alpha_{n-1},\dots,\alpha_{1})=f_{n-1}(\alpha_{n-1},\dots,\alpha_{1})=(f_1(\alpha_{n-1}),\dots,f_1(\alpha_{1})).$ Hence $(\beta_n,\dots,\beta_1)=(f_1(\alpha_n),\dots,f_1(\alpha_1)).$
It follows that $f$ is defined by $f_0$ and $f_1.$
 By definition we set $\F(c)=f_0(c)$ and $\F(\alpha)=f_1(\alpha)$. Prove that $\F$ is a functor i.e. it respects identity morphisms and composition.  Indeed $\F({\sf id}_c)=f_1({\sf id}_c)=f_1(s_0(c))=s_0(f_0(c))={\sf id}_{\F(c)}$ and $\F(\beta\alpha)=f_1(\beta\alpha)=f_1(d_1(\beta,\alpha))=d_1f_2(\beta,\alpha)=d_1(\F(\beta),\F(\alpha))=\F(\beta)\F(\alpha).$ It follows that $\F$ is the functor such that $f=\N\F.$

The isomorphism $\N(\C\times \D)\cong \N\C\times \N\D$ is obvious. The second isomorphism follows from $\N (\D^\C)={\sf Cat}(-,\D^\C)\cong {\sf Cat}(-\times \C,\D)\cong {\sf sSets}(\Delta[-]\times \N\C,\N\D)={\sf Map}(\N\C,\N\D).$
\end{proof}

\begin{Remark}
This lemma implies that we can consider ${\sf Cat}$ as a full subcategory in ${\sf sSets}$ identifying $\C$ with its nerve. Moreover this embedding respects products and inner hom's.
\end{Remark}

Let $\F,{\mathcal G}:\C\to \D$ be two functors and $\varphi:\F\to \mathcal G$ be a natural transformation. Then $\varphi\in \N({\sf Fun}(\C,\D))_1\cong {\sf Map}(\N\C,\N\D)_1={\sf Cat}(\Delta[1]\times \N\C,\N\D).$ Hence, the natural transfomation $\varphi:\F\to \mathcal G$ defines a simplicial homotopy between $\N\F$ and $\N\mathcal G$ that we denote by
\begin{equation}\label{eq_Nphi}
\N\varphi: \Delta[1]\times \N\C\longrightarrow \N\D,\hspace{1cm} \N\F\ \overset{\N\varphi}\sim\ \N\mathcal G.
\end{equation}

Here we use the symbol '$\sim$' for simplicial homotopy relation, however, it is not an equivalence relation for arbitrary simplicial sets.

\subsection{Classifying space of a category}

The {\bf classifying space} $\B\C$ of a category $\C$ is the geometric realisation of its nerve
$$\B\C=|\N\C|.$$
A category $\C$ is said to be {\bf contractible} if $\B\C$ is contractible.
If we consider a group $G$ as a category with one object, the classifying space $\B G$ is the classifying space of this group. This construction defines a functor
$$\B:{\sf Cat} \longrightarrow {\sf Top}.$$
Using \eqref{eq_prod} and \eqref{eq_Nphi} we get that if $\F,{\mathcal G}:\C\to\D$ are functors, then a natural transformation $\varphi:\F\to {\mathcal G}$ defines a homotopy
\begin{equation}\label{eq_homotopy_B}
\B\varphi: I\times \B\C\longrightarrow \B\D,\hspace{1cm} \B \F\ \overset{\B\varphi}\sim\ \B\mathcal G.
\end{equation}

\subsection{Classifying space and adjoint functors}

An {\bf adjunction} between categories $\C$ and $\D$ is a pair of functors $\F:\C \rightleftarrows \D:\G$ with a natural isomorphism
$$\D(\F c,d)\cong \C(c,\G d).$$ In this case the functor $\F$ is called a left adjoint functor, the functor $\G$ is called a right adjoint functor and the adjunction is denoted by $\F \dashv \G$.  Then there is an natural isomorphism $\D(\F c,\F c)\cong \C( c,\G\F c)$ and the image of ${\sf id}_{\F c}$ under this isomorphism is called the {\bf unit of adjunction} $\eta_c: c\to \G\F c.$ Similarly the isomorphism $\C(\G d,\G d)\cong \D(\F\G d,d)$ defines the {\bf counit of adjuntion} $\varepsilon_d:\F\G d\to d.$ These morphisms define a natural transformations
$$\eta : {\sf Id}_{\C} \longrightarrow \G\F, \hspace{1cm} \varepsilon :\F\G \longrightarrow {\sf Id}_{\D}.$$

\begin{Proposition}
If there is an adjunction between categories $\C$ and $\D$, then their classifying spaces $\B\C$ and $\B\D$ are homotopy equivalent.
\end{Proposition}
\begin{proof}
If $\F:\C\rightleftarrows \D :\G $ is an adjunction, then $\B \eta $ and $\B\varepsilon$ constructed in \eqref{eq_homotopy_B} are homotopies ${\sf id}_{\B \C}\sim(\B\G)(\B\F) $ and $(\B\F)(\B\G)\sim {\sf id}_{\B\D}.$
\end{proof}
\begin{Corollary}
If a category $\C$ has an initial  or a terminal object, it is contractible.
\end{Corollary}
\begin{proof}
Let $*$ be a category with one object and one morphism. Then $\C$ has a initial object if and only if there is an adjunction $\F:*\rightleftarrows \C:\G.$ Dually, it has a terminal object if and only if there is an adjunction $\F:\C\rightleftarrows *:\G.$
\end{proof}

\subsection{Classifying space of a category with pairwise coproducts}

A {\bf category  with pairwise coproducts} is a category $\C$ such that
for any objects  $c_1,c_2\in \mathcal{C}$ there exists the coproduct $c_1 \overset{i_1}{\longrightarrow} c_1\sqcup c_2 \overset{i_2}{\longleftarrow} c_2$ in $\mathcal{C}$ ($\mathcal{C}$ does not necessarily have an initial or terminal object). A category with pairwise products is defined dually.

\begin{Proposition}\label{lemma_coproduct_contractible} A category $\mathcal{C}$ with pairwise (co)products is contractible.
\end{Proposition}
\begin{proof} Assume that $\C$ is a category with pairwise coproducts.
 Chose an object $c_0\in
\mathcal{C}$ and consider the functor $\F:\mathcal{C}\to
\mathcal{C}$ given by the formula $\F(-)=-\sqcup c_0.$ By the
definition of coproduct $c\to c\sqcup c_0 \leftarrow c_0$, we obtain natural transformations  $\mathsf{Id}_\mathcal{C}\to \F \leftarrow \mathsf{Const}_{c_0}.$ Using \eqref{eq_homotopy_B}, we get that these natural transformations induce homotopies
$\mathsf{id}_{\mathsf{B}\mathcal{C}}\sim \B\F \sim \mathsf{const}_{c_0}.$
Hence, $\mathcal{C}$ is contractible. The case of a category with pairwise products is dual.
\end{proof}

\section{Representations of categories and higher limits.}\label{section_Representations_of_categories}

\subsection{Limits}

Recall \cite{Mac_Lane} that the limit  $\lim \: \mathcal F$ of a functor $\F:\C\to \mathcal D$ is an object of $\mathcal D$ together with a universal collection of morphisms $\{\varphi_c:\lim\: \F \to \F(c)\}_{c\in \C}$ such that  $\F(f) \varphi_c=\varphi_{c'}$ for any morphism $f:c\to c'.$ Universality means that for any object $d\in \mathcal D$ and any collection of morphisms $\{\psi_c:d\to \F(c)\}_{c\in \C}$ such that $\F(f) \psi_c=\psi_{c'}$ for any morphism $f:c\to c'$ there exists a unique morphism $\alpha:d\to \lim\:\F$ such that $\psi_c=\varphi_c\alpha.$ The limit is not always exists but if it exists it is unique up to unique isomorphism that commutes with morphisms $\varphi_c.$

Denote by ${\sf diag}:\D \to \D^\C$ the diagonal functor given by ${\sf diag}(d)(c)=d$ and ${\sf diag}(d)(\alpha)={\sf id}_d.$ For the sake of simplicity we will denote ${\sf diag}(d)$ by the symbol $d.$
If limits  of all functors $\C\to \D$ exist, they define the functor
$$\lim: \D^\C\longrightarrow \D$$ which is right  adjoint to the diagonal functor ${\sf diag} \dashv \lim$. In other words there is a natural isomorphism
$$\D^\C(d,\F)\cong \D(d,\lim\: \F).$$

\subsection{Limits over strongly connected categories}
Throughout the paper $k$ denotes a commutative ring.
Let $\F:\C\to \Mod(k)$ be a functor. Consider the subfunctor  ${\sf inv}\:\F\subseteq \F$ of 'invariant elements' given by the formula
$${\sf inv}\:\F(c)=\{x\in \F(c)\mid \forall c', \forall f_1,f_2:c\to c'\ \ \ \F(f_1)(x)=\F(f_2)(x)\}.$$

A category $\mathcal C$ is said to be {\bf strongly connected} if for any $c,c'\in \mathcal C$ the hom-set $ \mathcal C(c,c')$ is non-empty. For example, a category with one object is always strongly connected but the category of the ordered set $\{0<1\}$ is  not strongly connected because $\mathcal C(1,0)=\emptyset$.

\begin{Proposition}\label{Propostion_inv} Let $\C$ be a strongly connected category and $\F:\C\to  \Mod(k)$ be a functor. Then   $\lim\:\mathcal F$ exists and for any $c\in \C$ the morphism $\varphi_{c}:\lim\:\F\to \F(c)$ is a monomorphism and it induces the isomorphism
$$\lim\: \F\cong {\sf inv}\:\F(c).$$
\end{Proposition}
\begin{proof}
Fix an object $c_0\in \C$ and consider a morphism $f:c_0\to c.$
Since $\C$ is strongly connected, there exists a morphism $g:c\to
c_0.$ Using the definition of ${\sf inv}\:\F,$ we get ${\sf
inv}\:\F(f)={\sf inv}\:\F(g)^{-1}.$ Hence, ${\sf inv}\:\F(f)$ is
an isomorphism for any $f.$ Moreover, for any two morphisms
$f_1,f_2:c_0\to c$ we have ${\sf inv}\:\F(f_1)={\sf
inv}\:\F(g)^{-1}={\sf inv}\:\F(f_2).$ It follows that the
isomorphism   ${\sf inv}\:\F(c_0)\cong {\sf inv}\:\F(c)$ is
independent of the choice of $f.$ We denote by $\varphi_{c}:{\sf
inv}\:\F(c_0)\to  \F(c)$ the composition of this isomorphism with
the embedding $ {\sf inv}\:\F(c)\subseteq \F(c).$ It is easy to
see that $\F(f)\varphi_c=\varphi_{c'}$ for any morphism $f:c\to
c'.$ Prove the universal property of this collection of morphisms.
Assume that we have a module $M$ and a collection $\{\psi_c:M\to
\F(c)\}$ with the property $\F(f)\psi_c=\psi_{c'}$ for any
morphism $f:c\to c'.$ Thus
$\F(f_1)\psi_{c_0}=\psi_{c}=\F(f_2)\psi_{c_0}$ for any two
morphisms $f_1,f_2:c_0\to c,$ and hence, ${\rm
Im}\:\psi_{c_0}\subseteq {\sf inv}\:\F(c_0).$ Denote by
$\alpha:M\to {\sf inv}\:\F(c_0)$ the homomorphism induced by
$\psi_{c_0}.$ It is easy to check that $\alpha$ is the unique
homomorphism satisfying the property $\psi_c=\varphi_c\alpha.$
Hence $\lim\:\F={\sf inv}\:\F(c_0).$
\end{proof}

\begin{Example}
Let $G$ be a group considering as a category with one object. Then a functor $M:G\to {\sf Mod}(k)$ is a $kG$-module. Denote by $M^G$ submodule of invariants of $M.$ Then Proposition \ref{Propostion_inv} implies $$\lim\: M=M^G.$$
\end{Example}

\subsection{Higher limits of representations of categories.}
 A {\it representation} of a category $\mathcal{C}$  is
 a functor $\mathcal{F}:\mathcal{C}\to \mathsf{Mod}(k).$ Here we have to consider the category of 'big enough' modules. More presicily, if $\C$ is $U$-small for a Grothendieck universe $U,$ then by ${\sf Mod}(k)$ we denote the category of $U$-small modules.  It is well-known that  the category of representations $\mathsf{Mod}(k)^\mathcal{C}$ is an abelian category with enough projective and injective objects, and a sequence in this category is exact if and only if it is exact objectwise.

For any $c\in\C$ we denote by $P_{c}$ the representation given by $P_{c}=k[\C(c,-)],$
where $k[-]$ denotes the functor of free $k$-module.
The following isomorphism is a $k$-linear analogue of the Yoneda lemma:
\begin{equation}\label{eq_Yoneda_lemma}
{\sf Hom}_{{\sf Mod}(k)^\C}(P_c,\F)\cong \F(c).
\end{equation}
It follows that the functor ${\sf Hom}_{{\sf Mod}(k)^\C}(P_c,-)$ exact, and hence, $P_c$ is a projective representation.

All limits of all representations $\C\to {\sf Mod}(k)$ exist.
Therefore, we have a functor $\lim:{\sf Mod}(k)^\C\to {\sf Mod}(k)$ with the natural isomorphism:
\begin{equation} {{\sf Hom}_{{\sf Mod}(k)^\C}}(M,\mathcal{F})\cong \mathsf{Hom}_k(M,\lim\, \mathcal{F})
\end{equation} for any representation
$\mathcal{F}$ and any $k$-module $M.$ In particular, we have a natural isomorphism
\begin{equation}\label{eq_lim_hom}
\lim\: \F= {{\sf Hom}_{{\sf Mod}(k)^\C}}(k,\mathcal{F}).
\end{equation}

 Since $\lim$ is a right
adjoint functor, it is a left exact additive functor. Hence it has
right derived functors.
\begin{equation}\lim^n={\bf R}^n \: \lim.
\end{equation}
Using \eqref{eq_lim_hom} we get
\begin{equation} \lim^n \F=\mathsf{Ext}_{{\sf Mod}(k)^\C}^n(k,\mathcal{F}).
\end{equation}
It is well-known that if $G$ is a group and $M$ is a $k$-module, then $\lim^n M={\sf Ext}^n_{kG}(k,M)=H^n(G,M)=H^n(\B G,M).$ This can be generalised as follows
\begin{Proposition} Let $\C$ be a category and $M$ be a $k$-module. Then $$\lim^n M \cong H^n(\B \C,M).$$
\end{Proposition}
\begin{proof}
Denote by $E:\Delta^{\rm op}\to {\sf Sets}^\C,$ the simplicial functor whose $E_n(c)$ consists of sequences
$(\beta, \alpha_n,\dots,\alpha_1)$ with
 ${\sf dom}(\alpha_{i+1})={\sf codom}(\alpha_i),$
 ${\sf dom}(\beta)={\sf codom}(\alpha_n)$
 and ${\sf codom}(\beta)=c.$  Face maps act as follows $d_i(\beta,\alpha_n,\dots,\alpha_1)=(\beta,\alpha_n,\dots,\alpha_{i+1}\alpha_i,\dots, \alpha_1)$ for $1\leq i< n ,$ $d_0(\beta,\alpha_n,\dots,\alpha_1)=(\alpha_n,\alpha_{n-1},\dots,\alpha_2)$ and $d_n(\beta,\alpha_n,\alpha_{n-1},\dots,\alpha_1)=(\beta\alpha_n,\alpha_{n-1},\dots,\alpha_1).$ Degeneracies act as follows $s_i(\beta,\alpha_n,\dots,\alpha_1)=(\beta,\alpha_n,\dots,\alpha_{i+1}, {\sf id}_{c_i},\alpha_i, \dots, \alpha_1)),$ where $c_i={\sf codom}(\alpha_i).$ For $f:c\to c'$ the map $E_n(f):E_n(c)\to E_n(c')$ is given by $E_n(f)(\beta, \alpha_n,\dots,\alpha_1)=(f\beta, \alpha_n,\dots,\alpha_1).$ For any $n$ consider the map ${\tilde s}_{n+1}:E_n(c)\to E_{n+1}(c)$ given by ${\tilde s}_{n+1}(\beta,\alpha_n,\dots,\alpha_1)=({\sf id}_c,\beta,\alpha_n,\dots,\alpha_1).$ These maps satisfy the following identities $d_i{\tilde s}_{n+1}={\tilde s}_{n}d_i$ and $d_{n+1}{\tilde s}_{n+1}={\sf id},$ ${\tilde s}_{n+1}s_i=s_i{\tilde s}_n.$ By Proposition \ref{Prop_Contractible_criterion} we obtain that the simplicial set $|E(c)|$ is contractible for any $c.$
 It follows that the homology of the complex ${\sf C}_\bullet(E(c))\otimes k$ are concentrated in zero degree and $H_0({\sf C}_\bullet(E(c))\otimes k)=k.$ Note that there is an isomoprhism of representations:
 $${\sf C}_n(E)\otimes k=\bigoplus_{(\alpha_n,\dots,\alpha_1)\in (\N\C)_n} P_{{\sf codom}(\alpha_n)}.$$
Therefore, the complex presentations ${\sf C}_\bullet(E(-))$ is a projective resolution of the trivial presentation $k.$ Moreover, using the formula
\eqref{eq_Yoneda_lemma} we get
$${\sf Hom}_{\Mod(k)^\C}({\sf C}_n(E)\otimes k,M)\cong \bigoplus_{(\alpha_n,\dots,\alpha_1)\in (\N\C)_n}M\cong {\sf Hom}({\sf C}_n(\N\C),M).$$
One can check that this isomorphism gives an isomorphism of complexes
$${\sf Hom}_{\Mod(k)^\C}({\sf C}_\bullet(E)\otimes k,M)\cong {\sf Hom}({\sf C}_\bullet(\N\C),M).$$
Then applying homology $H^n$ to both sides we get
$\lim^n M\cong H^n(\N\C,M).$ Using \eqref{eq_homology}, we obtain the required statement.
\end{proof}

\begin{Corollary}
If a category $\C$ is contractible and $M$ is a $k$-module then $$\lim^i M=0$$ for $i>0.$
\end{Corollary}

\subsection{Preserving of higher limits.}
A category $\C$ is said to be $k${\it -acyclic}
 if the reduced homology groups $\tilde H_*(\mathsf{B}\mathcal{C},k)$ vanish. Of course, a contractible category is $k$-acyclic.

Let  $\Phi:\mathcal{C}\to \mathcal{D}$ be a functor. For an object
$d\in D$ we denote by $\Phi\!\!\downarrow\!\!d$ the comma category
(see \cite{Mac_Lane}) whose  objects are pairs $(c,\alpha)$, where
$c\in \mathcal{C}$  and $\alpha\in \mathcal{D}(\Phi(c),d)$. A
morphism in the comma category from  $(c,\alpha)$ to
$(c',\alpha')$ is given by $f:c\to c'$, which satisfies $\alpha'
\Phi(f)=\alpha.$ If $\mathcal{C}=\mathcal{D}$ and
$\Phi=\mathsf{Id}_\mathcal{D}$, the corresponding comma category
is denoted by $\mathcal{D}\!\!\downarrow\!\!d.$

\begin{Proposition}\label{proposition_A}{\rm (see \cite[5.4]{Jackowski-Slominska}).}
Let $\Phi:\mathcal{C}\to \mathcal{D}$ be a functor such that $\Phi\!\!\downarrow\!\!d$ is $k$-acyclic for any
$d\in D.$ Then for any representation $\mathcal{F}$  of
$\mathcal{D}$ there is an isomorphism
\begin{equation}\theta_{\Phi}:\lim^n\mathcal{F}
\overset{\cong}{\longrightarrow}  \lim^n(\mathcal{F} \Phi),
\end{equation}
which is natural in $\mathcal{F}$ and $\Phi.$
\end{Proposition}
\begin{Remark}
In the previous proposition by naturality in $\Phi$ we mean
that for a natural transformation $\alpha:\Phi_1\to \Phi_2$ of
functors satisfying the condition of the proposition we have
$\alpha_*\theta_{\Phi_1}=\theta_{\Phi_2},$ where
$\alpha_*:\lim^n(\mathcal{F}\Phi_1)\to
\lim^n(\mathcal{F}\Phi_2)$ is the map induced by $\alpha.$ In
\cite{Jackowski-Slominska} the fact that the isomorphism is natural is not stated explicitly, but it follows easily from the proof and
the fact that the obvious map $\theta^0_\Phi:\lim\:\mathcal{F}\to
\lim(\mathcal{F} \Phi)$ is natural in $\mathcal{F}$ and $\Phi.$
\end{Remark}
\begin{Corollary}\label{corollary_A_naturalness}
Let $\Phi_1,\Phi_2:\mathcal{C}\to \mathcal{D}$ be functors
satisfying the condition of proposition \ref{proposition_A} and
$\alpha:\Phi_1\to \Phi_2$ be a natural transformation. The induced
homomorphism $\alpha_*:\lim^n(\mathcal{F}\Phi_1)\to
\lim^n(\mathcal{F}\Phi_2)$ is an isomorphism.
\end{Corollary}
\begin{proof}
Since the isomorphism in proposition \ref{proposition_A} is
natural in $\Phi,$ there is a commutative diagram
$$\xymatrix{
 && \lim^n\mathcal{F}\ar[dll]_{\theta_{\Phi_1}}^\cong\ar[drr]^{\theta_{\Phi_2}}_\cong  && \\
\lim^n(\mathcal{F}\Phi_1)\ar[rrrr]^{\alpha_*} && &&
\lim^n(\mathcal{F}\Phi_2). }$$ Hence, $\alpha_*$ is an
isomorphism.
\end{proof}

\begin{Lemma}\label{Lemma_coproduct_projections}
Let $\mathcal{F}$ be a representation of a category $\C$ with coproducts. Then the morphisms $i_1,i_2:c\to c\sqcup c$ induce isomorphsms
\begin{equation}(i_k)_*:\lim^n\: \mathcal{F}(c) \overset{\cong}{\longrightarrow} \lim^n\:  \mathcal{F}(c\sqcup c).
\end{equation}
\end{Lemma}
\begin{proof} We denote by ${\sf sq}$ the functor $\mathcal C \to \mathcal C$ given by $c\mapsto c\sqcup c.$
An object of the comma category $\mathsf{sq}\!\downarrow\! c_0$ is a pair $(c,\alpha:c\sqcup c\to c_0).$ An arrow $\alpha:c\sqcup c\to c_0$ is defined by the pair of arrows $\alpha_1=\alpha\circ i_1:c\to c_0$ and $\alpha_2=\alpha\circ i_2:c\to c_0.$ Thus the category $\mathsf{sq}\! \downarrow\! c_0$ is isomorphic to the category, whose objects are triples $(c,\alpha_1:c\to c_0,\alpha_2:c\to c_0)$ and morphisms $f:(c,\alpha_1,\alpha_2)\to (c',\alpha_1',\alpha_2')$ are morphisms $f:c\to c'$ such that $\alpha_1'f=\alpha_1$ and $\alpha_2'f=\alpha_2.$ It is easy to see that the object $(c\sqcup c',\alpha_1+\alpha_1',\alpha_2+\alpha_2')$ is the coproduct of objects $(c,\alpha_1,\alpha_2)$ and $(c',\alpha_1',\alpha_2')$ in this category. Hence the category  $\mathsf{sq}\!\downarrow\! c_0$ has coproducts and by Lemma \ref{lemma_coproduct_contractible} it is contractible. Therefore, the functor ${\sf sq}$ satisfies the condition of Proposition \ref{proposition_A}. Finally, applying corollary \ref{corollary_A_naturalness} to the natural transformation $i_k:{\sf Id}_{\mathcal{C}}\to {\sf sq},$ we obtain the claimed isomorphism.
\end{proof}

\subsection{Monoadditive representations.}

 Let $\mathcal{C},\D$ be categories with pairwise coproducts. For $c_1,c_2\in \C$ we denote by $i_k=i_k(c_1,c_2):c_k\to c_1\sqcup c_2$  the universal morphisms $c_1 \overset{i_1}\to c_1\sqcup c_2 \overset{i_2}\leftarrow c_2$.
 For any functor $\mathcal{F}:\mathcal{C}\to \D$ there is a
morphism
\begin{equation}\label{eq_Fi1_Fi2}(\mathcal{F}(i_1),\mathcal{F}(i_2)):\mathcal{F}(c_1)\sqcup \mathcal{F}(c_2) \longrightarrow \mathcal{F}(c_1\sqcup c_2),
\end{equation}
 The functor
$\mathcal{F}$  is said to be {\bf additive} (resp.
{\bf monoadditive, split monoadditive}) if \eqref{eq_Fi1_Fi2}  is an
isomorphism (resp. monomorphism, split monomorphism in the category of bifunctors).
Consider the functors  ${\sf sq}:\C\to \C$ and ${\sf sq}:\D\to \D$ given by ${\sf sq}(x)=x\sqcup x.$ Then \eqref{eq_Fi1_Fi2} gives a natural transformation
$${\sf T}_\F:{\sf sq}\circ\F \longrightarrow \F\circ {\sf sq}.$$

Let $\D=\Mod(k).$ Starting from a monoadditive representation $\mathcal F:\C\to \Mod(k)$ we construct the representation
$\Sigma\mathcal{F}:=\mathrm{coker}({\sf T}_{\mathcal F})$ and get the short exact sequence:
\begin{equation}0\longrightarrow \mathcal{F}\oplus \mathcal F\overset{\sf T_{\mathcal F}}\longrightarrow \mathcal{F}\circ {\sf sq} \longrightarrow \Sigma\mathcal{F}\longrightarrow 0.
\end{equation}
A monoadditive representation $\F$ is said to be {\bf $\boldsymbol{n}$-monoadditive}, if $\Sigma \F$ is $(n-1)$-monoadditive, where $1$-monoadditive representation is just a monoadditive representation. For an $n$-monoadditive representation $\F$ we can define $\Sigma^m \F$ for $m\leq n$ by induction. If a representation $\F$ is $n$-monoadditive for any $n\geq 1,$  $\F$ is called {\bf $\boldsymbol{\infty}$-monoadditive}. For example, additive representations are $\infty$-monoadditive.

\begin{Proposition}\label{proposition_cokernel_limits}
Let $\mathcal{C}$ be a category with pairwise coproducts and $\mathcal{F}$
be a monoadditive representation of $\mathcal{C}.$ Then there is
an isomorphism
\begin{equation}\lim^n\:\mathcal{F}\cong \lim^{n-1}\: \Sigma\mathcal{F}
\end{equation}
for any $n\geq 0.$
\end{Proposition}

\begin{proof}
Consider the long exact sequence of higher limits assotiated with
the short exact sequence $0\longrightarrow
\mathcal{F} \oplus \mathcal{F}
\longrightarrow \mathcal{F}\circ{\sf sq} \longrightarrow
\Sigma \mathcal{F}\longrightarrow 0.$ Using lemma
\ref{Lemma_coproduct_projections}, we obtain that the homomorphism
$((i_1)_*,(i_2)_*):\lim^n\:\mathcal{F}\oplus
\lim^n\:\mathcal{F}\longrightarrow
\lim^n\mathcal{F}\circ {\sf sq}$ is an epimorphism, and hence the
map $\lim^n\mathcal{F}\circ{\sf sq}\longrightarrow
\lim^n\:\Sigma\mathcal{F}$  vanishes. Therefore  we get the short exact sequences
\begin{equation}0\longrightarrow \lim^n\: \Sigma\mathcal{F}\longrightarrow \lim^{n+1}\:\mathcal{F}\oplus \lim^{n+1}\:\mathcal{F}\longrightarrow \lim^{n+1}\:\mathcal{F}\circ{\sf sq} \longrightarrow 0,
\end{equation} which are the totalisations of the bicartesian squares
$$\xymatrix{
\lim^n\:\Sigma\mathcal{F} \ar[rr]^{s_1}\ar[d]^{s_2} && \lim^{n+1}\:\mathcal{F}\ar[d]^{(i_1)_*}_{\cong} \\
\lim^{n+1}\:\mathcal{F}\ar[rr]^{(i_2)_*}_\cong &&
\lim^{n+1}\mathcal{F}\circ{\sf sq}.}$$ By Lemma
\ref{Lemma_coproduct_projections} $(i_1)_*$ and $(i_2)_*$ are
isomorphisms, and hence $s_1$ and $s_2$ are isomorphisms.
\end{proof}

\begin{Corollary}\label{corollary_monoadditive}
If $\mathcal{F}$ is an $n$-monoadditive representation of a category
with pairwise coproducts, then $\lim^i\:\mathcal{F}=0$ for $0\leq i<n$ and $\lim^i\:\F=\lim^{i-n} \Sigma^n\F.$
\end{Corollary}
\begin{Corollary}\label{Corollary_infty_monoadditive}
If $\mathcal{F}$ is an $\infty$-monoadditive representation of a category with pairwise
coproducts, then $\lim^i\mathcal{F}=0$ for any $i\geq 0.$
\end{Corollary}

\begin{Lemma}\label{Lemma_split_mono_properties}
Let $\C,\D$ be categories with pairwise coproducts and $\F:\C\to \D$ be a split monoadditive functor. Then:
\begin{enumerate}
\item if $\Phi:\D\to {\mathcal E}$ is an additive functor, then $\Phi\circ\F$ is a split monoadditive functor;
\item if ${\mathcal F}':\C\to \D$ is a retract of $\F,$ then ${\mathcal F}'$ is a split monoadditive functor.
\end{enumerate}
\end{Lemma}
\begin{proof}
The proof is obvious.
\end{proof}

\begin{Proposition}
If $\F$ is a split monoadditive representation, then $\Sigma\F$ is a split monoadditive representation.
\end{Proposition}
\begin{proof} Let  $r_{c_1,c_2}:\F(c_1\sqcup c_2) \to \F(c_1)\sqcup \F(c_2) $ be a retraction of the natural transformation $(\F(i_1),\F(i_2)).$ Then $r_{c_1\sqcup c_1,c_2\sqcup c_2}$ is a retraction for the natural transfotmation
$((\F\circ {\sf sq})(i_1),(\F\circ {\sf sq})(i_2)).$ It follows
that $\F\circ {\sf sq}$ is a split monoadditive functor. The short
exact sequence ${\sf sq}\circ \F \mono \F\circ {\sf sq} \epi
\Sigma\F$ splits in the category of representations. It follows
that $\Sigma\F$ is a retract of $\F\circ{\sf sq}.$ Then Lemma
\ref{Lemma_split_mono_properties} implies that $\Sigma\F$ is a
split monoadditive representation.
\end{proof}
\begin{Corollary}
A split monoadditive representation is $\infty$-monoadditive.
\end{Corollary}
Hence, we have the following gradation:
$$\text{ split monoadditive }\Rightarrow\ \infty\text{-monoadditive } \Rightarrow \text{ monoadditive.}$$

Let $\mathcal O:\C\to {\sf Alg}(k)$ be a functor to  the category of $k$-algebras. We say that a representation $\mathcal F$ is a right (resp. left) {\bf $\boldsymbol{\O}$-module}  if $\F(c)$ endowed with a structure of right (resp. left) $\O(c)$-module for any $c\in \C$ such that $\F(f):\F(c)\to \F(c')$ is a $\O(c)$-module homomorphism, where the structure of $\O(c)$-module on $\F(c')$ is given by the homomorphism $\O(f):\O(c)\to \O(c').$ An $\O$-module $\mathcal M$ is said to be {\bf flat} if $\mathcal M(c)$ is flat over $\mathcal O(c)$ for any $c\in \C.$

\begin{Proposition}\label{Proposition_infty_mono_tensor}
Let $\O:\C\to {\sf Alg}(k)$ be an additive functor, ${\mathcal M}:\C\to {\sf Mod}(k)$ be a right $\O$-module such that $\mathcal M(c)$ is
 flat $\O(c)$-module for any $c\in \C$ and ${\mathcal N}:\C\to {\sf Mod}(k)$ be a left $\O$-module  which is $n$-monoadditive as a representation for $1\leq n\leq \infty$. Then $\mathcal M\otimes_\O \mathcal N$ is an $n$-monoadditve representation.
\end{Proposition}
\begin{proof}
Since $\O$ is additive $\mathcal O(c_1\sqcup c_2)\cong \mathcal O(c_1)\sqcup \mathcal O(c_2).$ Then the structure of $\mathcal O(c_i)$-module on $\mathcal N(c_i)$ induces a structure of  $\O(c_1\sqcup c_2)$-module on $\mathcal N(c_1)\oplus \mathcal N(c_2)$ and the morphism $\mathcal N(c_1)\oplus \mathcal N(c_2) \to \mathcal N(c_1\sqcup c_2)$ is a $\mathcal O(c_1\sqcup c_2)$-module homomorphism. Hence $\Sigma\mathcal N$ is a $\O\circ {\sf sq}$-module.  Since $\mathcal M(c_1\sqcup c_2)$ is a flat $\mathcal O(c_1\sqcup c_2)$-module, the morphism
\begin{equation}\label{eq_Mmmmmm}
\mathcal M(c_1\sqcup c_2)\otimes_{\O(c_1\sqcup c_2)}(\mathcal N(c_1)\oplus \mathcal N(c_2)) \to \mathcal M(c_1\sqcup c_2)\otimes_{\O(c_1\sqcup c_2) } \mathcal N(c_1\sqcup c_2)
\end{equation}
is a monomorphism. Moreover, we have the isomorphism
$$\mathcal M(c_1\sqcup c_2)\otimes_{\O(c_1\sqcup c_2)}(\mathcal N(c_1)\oplus \mathcal N(c_2))\cong \Big(\mathcal M(c_1)\otimes_{\O(c_1)}\mathcal N(c_1)\Big)\oplus \Big(\mathcal M(c_2)\otimes_{\mathcal O(c_2)}\mathcal N(c_2)\Big).$$
It follows that $\mathcal M\otimes_\O \mathcal N$ is a monoadditive representation and
$$\Sigma(\mathcal M\otimes_\O \mathcal N)(c)=\mathcal M(c\sqcup c)\otimes_{\mathcal O(c\sqcup c)} \Sigma\mathcal N(c).$$ Then by induction we get the required statement.
\end{proof}

\subsection{Spectral sequence of higher limits.}
\begin{Proposition} Let $\Phi:\mathcal A\to \mathcal B$ be a left exact functor between abelian categories with enough injectives and $a^\bullet\in {\sf Com}^+(\mathcal A)$ be a bounded below complex with $\Phi$-acyclic cohomologies.
Then there exists a cohomological spectral sequence $E$ of objects of $\mathcal B$ so that
\begin{equation} E\Rightarrow \Phi(H^n(a^\bullet)) \ \ \text{ and } \  \ E_1^{pq}={\bf R}^q \Phi(a^p).
\end{equation}
\end{Proposition}
\begin{proof}
Consider an injective Cartan-Eilenberg resolution $I^{\bullet \bullet}$ of $a^\bullet$ with an injection
$a^\bullet \mono I^{\bullet,0}$
 and differentials $d_{\rm I}^{pq}:I^{pq}\to I^{p+1,q}$
 and $d_{\rm II}^{pq}:I^{pq}\to I^{p,q+1}.$
Then there are two spectral sequences ${}_{\rm I}E$ and $_{\rm
II}E$ associated with the bicomplex
$B^{\bullet\bullet}=\Phi(I^{\bullet\bullet})$ which converge to
the cohomology of the totalisation $H^n({\sf
Tot}(B^{\bullet\bullet}))$ so that ${}_{\rm I}E_1^{pq}=H^q_{\rm
II}(B^{p\bullet}),$ ${}_{\rm I}E_2^{pq}=H^p_{\rm I}(H^q_{\rm
II}(B^{\bullet\bullet}))$ and  ${}_{\rm II}E_1^{pq}=H^p_{\rm
I}(B^{\bullet, q}),$ ${}_{\rm II}E_2^{pq}=H^q_{\rm II}(H^p_{\rm
I}(B^{\bullet\bullet})).$ Since $I^{\bullet\bullet}$ is a
Cartan-Eilenberg resolution, the complex $I^{\bullet,q}$ is
homotopy equivalent to the complex of its cohomologies $H_{\rm
I}^*(I^{\bullet,q})$ with zero differentials, and hence $H_{\rm
I}^p(B^{\bullet,q })=\Phi(H^p_{\rm I}(I^{\bullet,q})).$  The
complex $H^p_{\rm I}(I^{\bullet\bullet})$ is an injective
resolution of $H^p(a^\bullet).$ It follows that $H^q_{\rm
II}(H^p_{\rm I}(B^{\bullet\bullet}))={\bf
R}^q\Phi(H^p(a^\bullet)).$ Using the fact that the objects
$H^p(a^\bullet)$ are $\Phi$-acyclic, we obtain
\begin{equation}{}_{\rm II}E_2^{pq}=\left\{\begin{array}{l l}
0, & q\ne 0 \\
\Phi(H^p(a^\bullet)), & q=0.
\end{array}\right.\end{equation} Hence we get $H^n({\sf Tot}(B^{\bullet\bullet}))=\Phi(H^n(a^\bullet)).$ Therefore the spectral sequence $E:={}_{\rm I}E$ converges to $\Phi(H^n(a^\bullet)),$ and since $I^{p,\bullet}$ is an injective resolution of $a^p$, we have $E_1^{pq}=H_{\rm II}^q(\Phi(I^{p,\bullet}))={\bf R}^q\Phi (a^\bullet).$
\end{proof}
\begin{Corollary}\label{corollary_spectral} Let
$\mathcal F^\bullet$ be a bounded below complex of representations
of a category $\mathcal C$ with $\lim$-acyclic cohomologies.
Then there exists a cohomological  spectral sequence $E$ such that
\begin{equation}E \Rightarrow \lim\: H^n(\mathcal F^\bullet) \ \ \text{ and } \ \ E_1^{pq}=\lim^q \mathcal F^p.
\end{equation}
\end{Corollary}

\section{Limits over ${\sf Pres}(G)$}\label{section_higher_limits_pres_G}

\subsection{The category of presentations of a group}

In this section, we assume that $k=\Z$ and we consider the category ${\sf Pres}(G)$
of free presentations of a group $G$. The objects of ${\sf Pres}(G)$ are surjective homomorphisms $\pi:F\epi G$ where $F$ is
a free group and  morphisms $f:(\pi_1:F_1\epi G)\to (\pi_2:F_2\epi
G)$ are homomorphisms $f:F_1\to F_2$ such that $\pi_1=\pi_2f$. The
category ${\sf Pres}(G)$ has coproducts given by
$$(\pi_1:F_1\epi G)\sqcup(\pi_2:F_2\epi G)=((\pi_1,\pi_2):F_1*F_2\epi
G),$$ and hence it is contractible. In particular, for any abelian
group $A$ higher limits $\lim^i A$ vanish for $i>0.$ It is easy to
check that the category ${\sf Pres}(G)$ is strongly connected, and
hence, the limit of a functor $\mathcal F:{\sf Pres}(G)\to {\sf
Ab}$ coincides with the set of invariant elements of $\mathcal
F(\pi)$ for any presentation $\pi$. All the limits considered
bellow in this section are taken over the category ${\sf
Pres}(G)$.

 We will always denote by $R$ the kernel of an epimorphism $\pi:F\epi G.$ Therefore an object of ${\sf Pres}(G)$ defines a short exact sequence of groups
$$1\longrightarrow R \longrightarrow F \overset{\pi}{\longrightarrow} G \longrightarrow 1,$$
and $R$ can be considered as a functor $R:{\sf Pres}(G)\to {\sf Gr}.$

\subsection{Representations $\mathcal M\otimes_{\Z [F]} \f \otimes_{\Z [F]} \mathcal N$}

Consider the fucntor $\Z [F]:{\sf Pres}(G)\to {\sf Rings}$ that
sends $(F\epi G)$ to $\Z [F].$ Then the representation
$$\f=I(F):{\sf Pres}(G)\longrightarrow {\sf Ab}$$
has the structure of $\Z [F]$-bimodule. So if $\mathcal M:{\sf
Pres}(G)\to {\sf Ab}$ is a right $\Z [F]$-module and $\mathcal
N:{\sf Pres}(G)\to {\sf Ab}$ is a left $\Z [F]$-module, we can
define a tensor product:
$$\mathcal M\otimes_{\Z [F]} \f \otimes_{\Z [F]} \mathcal N:{\sf Pres}(G)\longrightarrow {\sf Ab}.$$

\begin{Proposition}\label{Proposition_I(F)tesor} Let $\mathcal M$ be a right $\Z [F]$-module and $\mathcal N$ be a left $\mathcal O$-module. Then the representations $\mathcal M\otimes_{\Z [F]}\f $ and $\f\otimes_{\Z [F]} \mathcal N$ are split monoadditive. Moreover, if one of $\mathcal M,\mathcal N$ is a flat $\Z [F]$-module, then the representation $\mathcal M\otimes_{\Z [F]} \f \otimes_{\Z [F]} \mathcal N$ is $\infty$-monoadditive, and hence,  for any $i\geq 0$
$$\lim^i\: \mathcal M\otimes_{\Z [F]} \f \otimes_{\Z [F]} \mathcal N =0.$$
\end{Proposition}
\begin{proof}
First we note that since the category ${\sf Pres}(G)$ is strongly
connected, then for any two presentations $\pi_1:F_1\epi G$ and
$\pi_2:F_2\epi G$ the embeddings $\pi_k \to  \pi_1\sqcup \pi_2$
split. Indeed, for any morphism $f:\pi_2\to \pi_1$ the morphism
$({\sf id}_{\pi_1},f):\pi_1\sqcup \pi_2\to \pi_1$ is a retraction
of $i_1:\pi_1\to \pi_1\sqcup \pi_2.$ It follows that for any
representation $\F$ the monomorphism $\F(\pi_k)\to\F(\pi_1\sqcup
\pi_2)$ splits.

It is well known \cite[Prop. 6.2.9]{Weibel} and easy to prove that
for any groups $G$ and $H$ there is an isomorphism of right
$\Z[G*H]$-modules $$I(G*H)\cong(I(G)\otimes_{\Z [G]}\Z[G*H])\oplus
(I(H)\otimes_{\Z [G]}\Z[G*H])$$ and the isomorphism of left
$\Z[G*H]$-modules $$I(G*H)\cong(\Z[G*H]\otimes_{\Z [G]}
I(G))\oplus(\Z[G*H]\otimes_{\Z [H]} I(H)).$$ It follows that
$$I(F_1*F_1)\otimes_{\Z[F_1*F_2]}\mathcal
N(\pi_1\sqcup\pi_2)=(I(F_1)\otimes_{\Z [F_1]} \mathcal
N(\pi_1\sqcup\pi_2))\oplus (I(F_2)\otimes_{\Z [F_2]} \mathcal
N(\pi_1\sqcup\pi_2)).$$ Since for any $k\in \{1,2\}$ the
monomorphism $\mathcal N(\pi_k)\to \mathcal N(\pi_1\sqcup\pi_2)$
splits, the monomirphism
$$I(F_k)\otimes_{\Z [F_k]}\mathcal N(\pi_k) \longrightarrow I(F_k)\otimes_{\Z [F_k]}\mathcal N(\pi_1\sqcup \pi_2)$$
splits too. It follows that $\f\otimes_{\Z [F]}\mathcal N$ is
split monoadditive. Similarly we get that $\mathcal M\otimes_{\Z
[F]}\f$ is split monoadditive. Then using the Proposition
\ref{Proposition_infty_mono_tensor} and Corollary
\ref{Corollary_infty_monoadditive} we get the required statement.
\end{proof}

\begin{Corollary}\label{Lemma_F_ab}
$\lim^i\: \f\otimes \F=0 $ and  $\lim^i\: F_{ab}\otimes \F=0$  for any functor $\F:{\sf Pres}\to {\sf Ab}$ and any $i\geq 0.$
\end{Corollary}
\begin{proof}
It follows from the equalities $\f\otimes \F=\f\otimes_{\Z [F]}(\Z
[F]\otimes \F)$ and $F_{ab}\otimes \F=\f\otimes_{\Z [F]}\F$.
\end{proof}

\subsection{Technical lemmas}

\begin{Lemma}\label{L1} If ${\bf a}, {\bf b}$ two right ideals of a ring $\Lambda$ which are free as right $\Lambda$-modules with bases $\{x_i\}$ and $\{y_j\},$ then ${\bf ab}$ is a free $\Lambda$-module with basis $\{x_iy_j\}.$
\end{Lemma}
\begin{proof}
Obvious.
\end{proof}

\begin{Corollary}\label{C1} Let $R_1,\dots, R_n$ be normal subgroups of $F,$ $\{y_{i,j}\}$ be a free basis of the group $R_i$ and $\r_i=(R_i-1)\Z [F].$ Then $\r_1\r_2\dots \r_n$ is a free right (or left) $F$-module and products $(y_{1,j_1}-1)\dots (y_{n,j_n}-1)$ form its basis.
\end{Corollary}

\begin{Lemma} If $R\mono F\epi G$ is a presentation of $G$ then
$$\dots \to  \frac{\f\r^2}{\f\r^3} \to  \frac{\r^2}{\r^3} \to  \frac{\f\r}{\f\r^2}  \to  \frac{\r}{\r^2} \to \frac{\f}{\f\r} \to  \frac{\Z [F]}{\r} \to  \Z \to 0$$
is a free resolution of the right trivial $G$-module $\Z,$ whose $2n$-th term is $\r^n/\r^{n+1}$, $(2n+1)$-st term is $\f\r^n/\f\r^{n+1}$ and morphisms are induced by embeddings.
\end{Lemma}
\begin{proof}
Using Corollary \ref{C1} we get that $\f\r^n$ and $\r^n$ are free
$F$-modules. Hence $\f\r^n/\f\r^{n+1}=\f\r^n\otimes_{\Z [F]}\Z G$
and $\r^n/\r^{n+1}=\r^n\otimes_{\Z [F]}\Z G$ are free $G$-modules.
Exactness is obvious.
\end{proof}

\begin{Lemma} \label{L3}
The map $r_1\otimes {\dots} \otimes r_n\mapsto (r_1-1)\dots (r_n-1)$ induces an isomorphism
$$R_{ab}^{\otimes n}\cong \r^n/\r^n\f.$$
\end{Lemma}
\begin{proof} Let $\{y_i\}$ is a free basis of $R.$
Corollary \ref{C1} implies that $\r^n$ is a free $F$-module and
products $(y_{i_1}-1)\dots (y_{i_n}-1)$ form its basis. It follows
that $\r^n/\r^{n}\f=\r^n\otimes_{\Z [F]}\Z$ is a free abelian
group and products $(y_{i_1}-1)\dots (y_{i_n}-1)$ form its basis.
$R_{ab}^{\otimes n}$ is a free abelian group with the basis given
by tensors $y_{i_1}\otimes {\dots} \otimes y_{i_n}.$ Hence the map
$r_1\otimes {\dots} \otimes r_n\mapsto (r_1-1)\dots (r_n-1)$
induces a bijection on the bases of free abelian groups, and
hence, it is an isomorphism.
\end{proof}

\begin{Lemma}\label{L4} If $\Lambda$ is a ring, $\a$ is its  right ideal and $\b$ is its left ideal, then there is a short exact sequence
$$0\longrightarrow {\sf Tor}_2^{\Lambda}(\Lambda/\a,\Lambda/\b)\longrightarrow \a\otimes_\Lambda \b \longrightarrow \a\b\longrightarrow 0,$$
where the map $\a\otimes_\Lambda\b\to \a\b$ is given by multiplication.
\end{Lemma}
\begin{proof}
Tensoring by $\b$ the short exact sequence $\a\mono \Lambda \epi \Lambda/\a$ we get the exact sequence
$$0 \longrightarrow {\sf Tor}_1^\Lambda(\Lambda/\a,\b) \longrightarrow \a\otimes_\Lambda\b\longrightarrow  \b \longrightarrow \b/\a\b \longrightarrow 0.$$
Since the kernel of $\b\to \b/\a\b$ is $\a\b$ we get the short exact sequence
$$0 \longrightarrow {\sf Tor}_1^\Lambda(\Lambda/\a,\b) \longrightarrow \a\otimes_\Lambda\b\longrightarrow \a\b \longrightarrow 0.$$ Using the long exact sequence of ${\sf Tor}_*^\Lambda(\Lambda/\a,-)$ with respect to the short exact sequence $\b\mono \Lambda\epi \Lambda/\b$ we obtain ${\sf Tor}_1^\Lambda(\Lambda/\a,\b)\cong {\sf Tor}_2^\Lambda(\Lambda/\a,\Lambda/\b).$
\end{proof}

\begin{Lemma}\label{L5} Let ${\bf a}$ and ${\bf b}$ be ideals of $\Z [F]$ such that ${\sf Tor}(\Z [F]/{\bf a},\Z [F]/{\bf b})=0.$ Then the map $x\otimes y\mapsto xy$ induces
$$\a \otimes_{\Z [F]} \b\cong \a\b.$$
\end{Lemma}
\begin{proof}
For any  $\mathbb Z[F]$-modules $M,N$ there is an isomorphism
$M\otimes_{\Z [F]}N=(M\otimes N)_F.$ Then there is a spectral
sequence of composition $H_*(F,{\sf Tor}_*(M,N))\Rightarrow {\sf
Tor}_*^{\mathbb Z[F]}(M,N).$ Using this spectral sequence we
obtain that ${\sf Tor}(\Z [F]/\a,\Z [F]/\b)=0$ and $H_2(F,-)=0$
imply ${\sf Tor}_2^{\Z [F]}(\Z [F]/\a,\Z [F]/\b)=0.$ Hence Lemma
\ref{L4} implies the claimed isomorphism.
\end{proof}

\begin{Lemma}\label{lemma49} Let ${\bf a}'\subset {\bf a}$ and ${\bf b'}\subset {\bf b}$ be ideals of $\Z [F]$ such that
${\sf Tor}(\Z [F]/{\bf a},\Z [F]/{\bf b})=0.$ Then the map $x\otimes y\mapsto xy$ induces
$$(\a/\a') \otimes_{\mathbb Z[F]} (\b/\b')\cong \a\b/(\a'\b+\a\b').$$
\end{Lemma}
\begin{proof}
Tensoring on $\b$ the short exact sequence $\a'\mono \a \epi \a/\a'$ we get the exact sequence
$$\a' \otimes_{\Z [F]} \b \longrightarrow \a\b \longrightarrow (\a/\a')\otimes_{\Z [F]} \b \longrightarrow 0.$$
The image of the map $\a'\otimes_{\Z [F]} \b \to \a\b$ is $\a'\b.$
Hence $(\a/\a')\otimes_{\Z [F]} \b \cong \a\b/\a'\b.$ Tensoring
the short exact sequence $\b'\mono \b\epi \b/\b'$ on $\a/\a'$ we
get the exact sequence
$$(\a/\a')\otimes_{\Z [F]} \b' \longrightarrow  \a\b/\a'\b \longrightarrow (\a/\a')\otimes_{\Z [F]} (\b/\b')\longrightarrow 0.$$
The image of the map $ (\a/\a')\otimes_{\Z [F]} \b'\to
\a\b/\a'\b'$ is $(\a'\b+\a\b')/\a'\b$ which implies the claimed
isomorphism.
\end{proof}

A {\bf functorial ideal} of $\Z [F]$ is a functor ${\bf x}:{\sf
Pres}(G)\to {\sf Ab}$ such that ${\bf x}$ is an ideal of $\Z [F]$
for any $\pi\in {\sf Pres}(G).$ For example, ${\bf x}=\r^n+\f^m.$

\begin{Lemma}\label{L6_5} If ${\bf x}$ is a functorial ideal of $\Z [F]$ such that $\Z [F]/{\bf x}$ is torsion free, then ${\bf x}$ is a flat $\Z [F]$-module.
\end{Lemma}
\begin{proof}
Using the spectral sequence $H_*(F,{\sf Tor}_*(M,N))\Rightarrow
{\sf Tor}^{\Z [F]}(M,N),$ we get $${\sf Tor}_1^{\Z [F]}({\bf
x},-)={\sf Tor}_2^{\Z [F]}(\Z [F]/{\bf x},-)=H_1(F,{\sf Tor}(\Z
[F]/{\bf x},-)).$$ So if $\Z [F]/{\bf x}$ is torsion free, then
${\bf x}$ is flat over $\Z [F].$
\end{proof}

\begin{Lemma}\label{L7} If ${\bf x},{\bf y}$ are functorial ideals of $\Z [F]$ and $\Z [F]/{\bf x}$ (or $\Z [F]/{\bf y}$) is torsion free as an abelian group for all $\pi\in {\sf Pres}(G)$, then for $i\geq 0$
$$\lim^i\: {\bf x}\f{\bf y} =0.$$
\end{Lemma}
\begin{proof}
It follows from Lemmas \ref{L5}, \ref{L6_5}  and Proposition \ref{Proposition_I(F)tesor}.
\end{proof}

\begin{Lemma}\label{L8} The abelian group $\Z [F]/\r^n$ is a free abelian group for any $n\geq 1.$
\end{Lemma}
\begin{proof}
It is enough to prove that  $\r^k/\r^{k+1}$ is free abelian.
 Since $\r^k$ is a free $F$-module, $\r^k/\r^{k+1}=\r^k\otimes_{\Z [F]}\Z G$ is a free $G$-module, and hence, it is a free abelian group.
\end{proof}

\begin{Corollary} For any $i,n\geq 0$ and a functorial ideal $ {\bf x}$ of $\Z [F]:$
$$\lim^i\: \r^n\f{\bf x}=0 \hspace{1cm} \text{ and }\hspace{1cm} \lim^i\:{\bf x}\f\r^n=0.$$
\end{Corollary}
\begin{proof}
Follows from Lemmas \ref{L7} and \ref{L8}.
\end{proof}

\begin{Lemma}
For any abelian group $A$
$$\lim^i\: \r^n\otimes A= \lim^i\: R_{ab}^{\otimes n}\otimes A =
\begin{cases} I^{\otimes n}\otimes A, & i=n,\\
0,& i\ne n. \end{cases}$$
\end{Lemma}
\begin{proof} Lemma \ref{L3} gives the short exact sequence:
$$0\longrightarrow R_{ab} \longrightarrow \f/\r\f \longrightarrow I \longrightarrow 0.$$
Since the short exact sequence  consists of free abelian groups,
tensoring it on $I^{\otimes k}\otimes R_{ab}^{\otimes n-k-1}\otimes A$ we get the short exact sequence
$$0\longrightarrow I^{\otimes k}\otimes R_{ab}^{\otimes n-k}\otimes A \longrightarrow P \longrightarrow I^{\otimes k+1}\otimes R_{ab}^{\otimes n-k-1}\otimes A
\longrightarrow 0, $$
where $$P=I^{\otimes k}\otimes  R_{ab}^{\otimes n-k-1}\otimes
A\otimes \f/\r\f=(I^{\otimes k}\otimes  R_{ab}^{\otimes
n-k-1}\otimes A\otimes \Z G)\otimes_{\Z [F]} \f.$$ Using
Proposition \ref{Proposition_I(F)tesor} we obtain $\lim^i\: P=0$
for all $i.$ Hence $$\lim^i I^{\otimes k}\otimes R_{ab}^{\otimes
n-k}\otimes A=\lim^{i-1} I^{\otimes k+1}\otimes R_{ab}^{\otimes
n-k-1}\otimes A.$$ By induction we get $\lim^i\:  R_{ab}^{\otimes
n}\otimes A =\lim^{i-n}\: \  I(G)^{\otimes n}\otimes A.$ It
follows that
$$ \lim^i\: R_{ab}^{\otimes n}\otimes A =
\begin{cases} I^{\otimes n}\otimes A, & i=n,\\
0,& i\ne n. \end{cases}$$
Now we prove that $\lim^i\: \r^n\otimes A=\lim^i\:R_{ab}^{\otimes n} \otimes A. $ Lemma \ref{L3} gives a short exact sequence
$\r^n\f\mono \r^n \epi R_{ab}^{\otimes n}.$ Since it consists of free abelian groups, we get the following short exact sequence:
$$0\longrightarrow\r^n\f\otimes A\longrightarrow \r^n\otimes A \longrightarrow R_{ab}^{\otimes n}\otimes A \longrightarrow 0.$$
Using Lemma \ref{L5} and Lemma \ref{L8} we get $\r^n\f\otimes
A=\r^n\otimes_{\Z [F]} \f\otimes_{\Z [F]} (\Z [F]\otimes A).$
Proposition \ref{Proposition_I(F)tesor} implies that $\lim^*
\r^n\f\otimes A=0.$
\end{proof}

\subsection{Group homology as higher limits}

\begin{Theorem}\label{theorem_homology} For any group $G$ and  abelian group $A$

$$H_{2n-i}(G,A)=\lim^i\: \r^n/(\f\r^n+\r^n\f) \otimes A, \hspace{1cm} 0\leq i \leq n-1,$$

$$ H_{2n+1-i}(G,A)=\lim^i\: \f\r^n/(\r^{n+1}+\f\r^n\f) \otimes A, \hspace{1cm} 0\leq i \leq n-1.$$
\end{Theorem}
\begin{proof}
We prove the first formula, the second formula can be proved similarly. Consider '$2n$-truncated Gruenberg resolution':
$$ {\sf t}_{2n}P: \ \  (
  0\to \r^n/\f\r^n \mono  \f\r^{n-1}/\f\r^n \to{\dots} \to     \r/\r^2 \to \f/\f\r \to \Z [F]/\r \to  0)$$
Then $H_i({\sf t}_{2n}P\otimes_{\Z G}A)=H_i(G,A)$ for $0\leq i\leq 2n.$
Since $A$ has trivial action of $G,$ we get $M\otimes_{\Z G}A=M/M\f\otimes A.$
Hence, ${\sf t}_{2n}P\otimes_{\Z G}A$ is equal to
$$ \r^n/(\f\r^n+\r^n\f)\otimes A \to   (\f\r^{n-1}/\f\r^{n-1}\f) \otimes A \to{\dots} \to     \r/\r\f\otimes A \to \f/\f^2\otimes A \to \Z [F]/\f\otimes A.$$
Apply the spectral sequence of higher limits to this complex:
$$ \lim^i\: {\sf t}_{2n}P_j\otimes_{\Z  G}A \Rightarrow H_{j-i}(G,A)  $$
The short exact sequence $$\f\r^k\f\otimes A\mono \f\r^k\otimes
A\epi (\f\r^k/\f\r^k\f)\otimes A$$ implies $\lim^i\:
(\f\r^k/\f\r^k\f)\otimes A=0.$ Remind $(\r^k/\r^k\f)\otimes
A=R_{ab}^{\otimes k}\otimes A.$ Set
$\F=(\r^n/(\f\r^n+\r^n\f))\otimes A.$ Then the first page of the
spectral sequence looks as follows:

\begin{footnotesize}
\begin{center}
\begin{tabular}{c|ccccccc}
$\vdots$ & $\vdots$    & $\vdots$& $\vdots$& $\vdots$ &  $\vdots$ & $\vdots$  & $\vdots$ \\
$n$ & $\lim^{n}\: \F$     & 0 & 0& $\dots$ &  0 & 0  & 0 \\
$n-1$ & $\lim^{n-1}\: \F$ & 0 & $I^{\otimes n-1}\otimes A$ & $\dots$ & 0 & 0 & 0 \\
$\vdots$ & $\vdots$       & $\vdots$ & $\vdots$& $\dots$& $\vdots$ & $\vdots$ & $\vdots$ \\
$2$ & $\lim^2\: \F$       & 0 & 0& $\dots$ & 0 & 0& 0\\
$1$ & $\lim^1\: \F$       & 0 & 0& $\dots$&  $I\otimes A$ & 0 & 0\\
$0$ & $\lim^0\: \F$       & 0 & 0& $\dots$ & 0 & 0 & $A$ \\ \hline
  & $2n$          & $2n-1$  & $2n-2$ & $\dots$ & 2 & 1 & 0
\end{tabular}
\end{center}
\end{footnotesize}
By the lacunary reason we get the isomorphisms.
\end{proof}

\begin{Corollary}
For any group $G,$  abelian group $A$ and $n\geq 1$
$$H_{2n}(G,A)=\lim^0\: \r^n/(\f\r^n+\r^n\f) \otimes A,$$
$$ H_{2n+1}(G,A)=\lim^0\: \f\r^n/(\r^{n+1}+\f\r^n\f) \otimes A.$$
\end{Corollary}

\begin{Theorem}[\cite{Ivanov-Mikhailov}] For any group $G$ and  a $G$-module $M$ there are isomorphisms:
$$ H_{2n-i}(G,M)=\lim^i\: R_{ab}^{\otimes n}\otimes_{\Z G} M, \hspace{1cm} 0\leq i \leq n-1.$$
Moreover,
$$\lim^i\: R_{ab}^{\otimes n}\otimes_{\Z G} M=0, \hspace{1cm} n+1\leq i$$
and there is a short exact sequence:
$$ 0\longrightarrow H_n(G,M) \longrightarrow \lim^n\: R_{ab}^{\otimes n}\otimes_{\Z G}M \longrightarrow (I^{\otimes n-1}\otimes M)\cdot I\longrightarrow 0. $$
\end{Theorem}

\section{Hochschild and cyclic homology as higher limits}

In this section we state results that proved in \cite{Ivanov-Mikhailov} and \cite{Quillen2}. Let $k$ be a field and $A$ be an algebra. Consider the category ${\sf Pres}(A)$
of free presentations of an algebra $A$. The objects of ${\sf Pres}(A)$ are surjective homomorphisms $\pi:F\epi A$ where $F$ is a free algebra and morphisms $f:(\pi_1:F_1\epi A)\to (\pi_2:F_2\epi A)$ are homomorphisms $f:F_1\to F_2$ such that $\pi_1=\pi_2f$. The limits considered in this section are taken over this category.

\begin{Theorem}[\cite{Ivanov-Mikhailov}]\label{theorem_hochschild__}
For an algebra $A$ and an $A$-bimodule $M$, for $n\geq 1$,
there are natural isomorphisms
$$
H_{2n-i}(A,M)\simeq\lim^i\  (I^n/I^{n+1})\otimes_{A^e} M \ \text{ for } \  0\leq i<n.
$$
\end{Theorem}

For an $F$-bimodule $M$ we set
$$M_\natural=\frac{M}{[M,F]}=H_0(A,M),$$
where $[M,F]$ is the vector space generated by elements $mf-fm.$

\begin{Corollary}
For an algebra $A$ and $n\geq 1$,
there are natural isomorphisms
$$
HH_{2n-i}(A)\simeq\lim^i\  (I^n/I^{n+1})_\natural \ \text{ for } \  0\leq i<n.
$$
\end{Corollary}

\begin{Theorem}[Quillen {\cite{Quillen2}}]
Let $A$ be an  algebra over a field $k$ of characteristic
$0$. Then there are isomorphisms
\begin{equation}HC_{2n}(A) \cong \lim \ (F/I^{n+1})_\natural,
\end{equation}
\end{Theorem}

\begin{Theorem}[\cite{Ivanov-Mikhailov}]
Let $A$ be an augmented algebra over a field $k$ of characteristic
$0$. Then there are isomorphisms
\begin{equation}HC_{2n-1}(A) \cong \lim^1 \ (F/I^{n+1})_\natural.
\end{equation}
\end{Theorem}

\section{fr-codes of functors ${\sf Gr}\to {\sf Ab}$}

Denote by ${\sf Pres}$ the category whose objects are presentations of a group $F\overset{\pi}\epi G,$ and whose morphisms are commutative squares
$$\xymatrix{
F_1\ar[rr]\ar[d]^{\pi_1} & & F_2\ar[d]^{\pi_2}\\
G_1\ar[rr] & & G_2.
}$$
Consider the forgetful functor
$${\sf Pres}\longrightarrow {\sf Gr},$$
that sends the presentation $F\epi G$ to $G.$ The fibre of this
functor over a group $G$ is the category ${\sf Pres}(G).$ The
following ideals of $\Z [F]$
$$\f=I(F)=(F-1)\Z [F], \hspace{1cm} \r=(R-1)\Z [F]$$
depend on $F$ and
can be considered as functors $\f,\r:{\sf Pres}\to {\sf Ab}$. Using the operations of sum, product and intersection of ideals we can get a lot of different functors. For example
$$ \f\r^n+\r^n\f, \r^{n+1}+\f\r^n\f : {\sf Pres}\longrightarrow {\sf Ab}.$$
Denote by ${\sf ML}(\f,\r)$ the set of all functorial ideals of
$\Z [F]$ that can be received form $\r,\f$ using the operations of
sum, product and intersection, including $\Z[F]$ itself. Then for
any ${\bf x}\in {\sf ML}(\f,\r)$ and $i\geq 0$ we get the functor
$${}^i[{\bf x}] : {\sf Gr}\longrightarrow {\sf Ab}$$
given by
$${}^i[{\bf x}](G)= \lim^i_{{\sf Pres}(G)}\ {\bf x}.$$
Moreover, if ${\bf x}\subseteq {\bf y}$ then we get a morphism
${}^i[{\bf x}] \longrightarrow {}^i[{\bf y}].$ We will prove that ${}^0[{\bf x}]=0$ for any ${\bf x},$ and the most important case is $i=1.$ So we set
${}^1[{\bf x}]=[{\bf x}].$
\begin{Definition}
An {\bf fr-code} of a functor $\F: {\sf Gr}\to {\sf Ab}$ is an
isomorphism
$$\F\cong {}^i[{\bf x}].$$
  A {\bf polynomial ${\bf fr}$-code} is an ${\bf fr}$-code $\F\cong {}^i[{\bf x}],$ such that  ${\bf x}={\sum}_j\ \a_{j,1}\a_{j,2}{}\dots {}\a_{j,m_j},$
where $\a_{j,k} \in \{\f,\r\}.$
\end{Definition}

\begin{Example}
Consider the functor $G\mapsto I(G).$ There is a short exact
sequence $\r\mono \f \epi I(G).$ Using the long exact sequence of
higher limits and the fact that $\lim^i \f=0$, we get the equality
$\lim^i\: \r=\lim^{i-1}\: I(G).$ Thus we obtain a polynomial ${\bf
fr}$-code for $I(G):$
$$I(G)\cong [\r].$$
\end{Example}

The main point of this section is the fact that a lot of useful
functors ${\sf Gr}\to {\sf Ab}$ and transformations between them
have polynomial ${\bf fr}$-codes. In particular, we will prove the
following polynomial ${\bf fr}$-codes:

\begin{align*}
& I(G)/I(G)^n=[\r+\f^n],\\
& G_{ab}=[\r+\f^2],\\
& I(G)^{\otimes_{\Z [F]} n}=[\sum_{i=1}^n \f^{i-1}\r\f^{n-i}],\\
& G_{ab}^{\otimes n}=[\sum_{i=1}^n \f^{i-1}\r\f^{n-i}+\f^{n+1}],\\
& H_{2n}(G)=[\f\r^n+\r^n\f],\\
& H_{2n-1}(G)=[\r^n+\f\r^{n-1}\f],\\
& L_{n-i} \bigotimes^n G_{ab}={}^i[\r^n+\f^{n+1}],\\
& I(G)^{\otimes n}={}^n[\r^n].
\end{align*}
There are also non-polynomial codes for some functors:
$$I(G)^n=[\r\cap \f^n].$$

\subsection{fr-codes for homology of groups}
\begin{Proposition}
There are polynomial $\bf {fr}$-codes for functors of homology
$H_*:{\sf Gr}\to {\sf Ab}$ except
$H_2$:$$H_{2n+2}(G)\cong[\f\r^{n+1}+\r^{n+1}\f], \hspace{1cm}
H_{2n-1}(G)\cong[\r^{n}+\f\r^{n-1}\f]$$ for  $n\geq 1.$
\end{Proposition}
\begin{proof}
Theorem \ref{theorem_homology} implies $H_{2n+2}(G)\cong \lim^0\:
\r^{n+1}/(\f\r^{n+1}+\r^{n+1}\f)$ and $H_{2n+1}(G)=\lim^0\:
\f\r^n/(\r^{n+1}+\f\r^n\f).$ Applying the long exact sequence of
limits and using $\lim^{0,1}\: \r^{n+1}=0$ and $\lim^*\: \f\r^n=0$
we get the claimed isomorphisms.
\end{proof}

\subsection{fr-codes for $L_i\bigotimes^n G_{ab}$}

Recall that if $\F:{\sf Ab}\to {\sf Ab}$ and $A$ is abelian
simplicial group, then  Dold-Puppe derived functors $L_i\:\F:{\sf
Ab}\to {\sf Ab}$ can be defined using a free abelian simplicial
resolution $X_\bullet$ of $A$ as follows
$$L_i\F(A)=\pi_i(\F(X_\bullet)).$$
Here we consider the derived functors of the functor the functor of tensor power $\bigotimes^n:{\sf Ab}\to {\sf Ab}$ that sends $A$ to $A^{\otimes n}.$

Let $A_1,\dots,A_n$ are abelian groups. Consider a free resolution $P_i=(0\to P_{i,1}\to P_{i,0}\to 0 )$ of $A_i.$ Then ${\sf Tor}_i(A_1,\dots,A_n)$ is defined as follows
$${\sf Tor}_i(A_1,\dots,A_n)=H_i(P_1\otimes \dots \otimes P_n).$$ Eilenberg-Zilber Lemma implies that
$$L_i\:{\bigotimes}^n A={\sf Tor}_i(A,\dots,A).$$

\begin{Proposition} The functors $L_{i} \bigotimes^{\otimes n} G_{ab} $ has the following ${\bf fr}$-codes
$$L_{n-i}{\bigotimes}^nG_{ab}\cong {}^i[\r^n+\f^{n+1}]. $$
\end{Proposition}
\begin{proof}
Consider the resolution $\bar R \to F_{ab}$ of $G_{ab},$ where
$\bar R=R[F,F]/[F,F].$  Then \hbox{$H_i((\bar R\to
F_{ab})^{\otimes n})=L_i\bigotimes^n G_{ab}.$} Consider the
acyclic complex $(F_{ab}\to F_{ab})^{\otimes n}$ and cokernel of
the embedding $$(\bar R\to F_{ab})^{\otimes n} \mono (F_{ab} \to
F_{ab})^{\otimes n}\epi C_\bullet.$$ Then
$H_{i}(C_\bullet)=L_{i-1}\bigotimes^n G_{ab}.$ Consider the
spectral sequence of higher limits (Corollary
\ref{corollary_spectral}) $\lim^i C_k\Rightarrow
L_{k-i-1}\bigotimes^n G_{ab}.$ If $0\leq k<n,$ then $C_k$ is a
direct sum of quotients of $F_{ab}^{\otimes n}$ by an abelian
group of the form $A\otimes F_{ab}\otimes B.$  Using Lemma
\ref{Lemma_F_ab} and long exact sequences of higher limits, we get
$\lim^i\: C_k=0$ for all $i$ and $k\ne n.$ It follows that $\lim^i
C_n=L_{n-i-1}\bigotimes^nG_{ab}.$ Note that $F_{ab}=\f/\f^2,$
$\bar R=(\r+\f^2)/\f^2,$ and hence, $$F_{ab}^{\otimes n}\cong
\f^n/\f^{n+1},\ \bar R^{\otimes n}\cong
(\r^n+\f^{n+1})/\f^{n+1}.$$ Then $C_n=F_{ab}^{\otimes n}/\bar
R^{\otimes n}=\f^n/(\r^n+\f^{n+1}).$ Since ${}^i[\f^n]=0,$ we
obtain $\lim^i C_n={}^{i+1}[\r^n+\f^{n+1}].$
\end{proof}

\subsection{Other computations}

\begin{Proposition} For any $n\geq 1$ there is a polynomial ${\bf fr}$-code for $I^{\otimes_{\Z G} n}$:
$$I^{\otimes_{\Z G}n}\cong \left[{\sum}_{i=1}^n \f^{i-1}\r\f^{n-i}\right].$$
Moreover, ${}^i[\sum \f^{i-1}\r\f^{n-i}]=0$ for $i\ne 1.$
\end{Proposition}
\begin{proof}
Since $I\cong \f/\r,$  we get $I^{\otimes_{\Z [F]} n}\cong
\f^n/(\sum \f^{i-1}\r\f^{n-i}).$ Then using the short exact
sequence $\sum \f^{i-1}\r\f^{n-i}\mono \f^n \epi I^{\otimes n}$ we
obtain the required isomorphism.
\end{proof}
\begin{Proposition}
There is a polynomial ${\bf fr}$-code for the functor $I/I^n:$
$$I/I^n\cong [\r+\f^n].$$ Moreover, ${}^i[\r+\f^n]=0$ for $i\ne 1,$ and the projection $I\epi I/I^n$ is induced by the embedding $\r \subset \r+\f^n.$
 \end{Proposition}
\begin{proof}
This follows from the long exact sequence of higher limits applied to rows of the diagram
$$
\xymatrix{
0\ar[r] & \r\ar[r]\ar[d] & \f\ar[r]\ar@{=}[d] & I(G)\ar[r]\ar[d] & 0\\
0\ar[r] & \r+\f^n\ar[r] & \f \ar[r] & I(G)/I(G)^n\ar[r] & 0.
}$$
\end{proof}

\begin{Proposition} There is a polynomial ${\bf fr}$-code for the functor $G_{ab}^{\otimes n}:$
$$G_{ab}^{\otimes n}=[{\sum}_{i=1}^n \f^{i-1}\r\f^{n-i}+\f^{n+1}].$$
Moreover, ${}^i[{\sum}_{i=1}^n \f^{i-1}\r\f^{n-i}+\f^{n+1}]=0$ for $i\ne 1,$ and the map $I(G)^{\otimes n} \to G_{ab}^{\otimes n}$ is induced by the embedding ${\sum}_{i=1}^n \f^{i-1}\r\f^{n-i}\subset {\sum}_{i=1}^n \f^{i-1}\r\f^{n-i}+\f^{n+1}.$
\end{Proposition}
\begin{proof}
Since $G_{ab}=\f/(\r+\f^2),$ we get $G_{ab}^{\otimes n}=\f^n/({\sum}_{i=1}^n \f^{i-1}\r\f^{n-i}+\f^{n+1}).$ Then the statement follows from the diagram
$$\xymatrix{
0\ar[r] & {\sum}_{i=1}^n \f^{i-1}\r\f^{n-i} \ar[r]\ar[d] & \f^n\ar[r]\ar@{=}[d] & I(G)^{\otimes_{\Z [F]}n}\ar[r]\ar[d] & 0\\
0\ar[r] & {\sum}_{i=1}^n \f^{i-1}\r\f^{n-i}+\f^{n+1}\ar[r] & \f^n\ar[r] & G_{ab}^{\otimes n}\ar[r] & 0.
}$$
\end{proof}

Now lets consider more non-standard example, namely that mentioned
in the introduction. There is a ${\bf fr}$-code
\begin{equation}\label{theiso}
{\sf Tor}(G_{ab}\otimes G_{ab}, G_{ab})=[{\bf rfr+frr+ffff}].
\end{equation}
To see (\ref{theiso}), observe first that these is a natural
isomorphism
$$
G_{ab}^{\otimes 2}=\frac{{\bf f}^2}{{\bf rf+fr+f}^3}
$$
Hence, there is a short exact sequence with free abelian two left
terms:
$$
0\to \frac{{\bf rf+fr+f}^3}{{\bf f}^3}\to \frac{{\bf f}^2}{{\bf
f}^3}\to G_{ab}^{\otimes 2}\to 0.
$$
Tensoring with $G_{ab}=\frac{\bf f}{{\bf r+f}^2},$ we get the
exact sequence
$$
0\to {\sf Tor}(G_{ab}^{\otimes 2}, G_{ab})\to \frac{{\bf
rf+fr+f}^3}{{\bf f}^3}\otimes\frac{\bf f}{{\bf r+f}^2}\to
\frac{{\bf f}^2}{{\bf f}^3}\otimes G_{ab}\to G_{ab}^{\otimes 3}\to
0
$$
By Lemma \ref{lemma49},
$$
\frac{{\bf rf+fr+f}^3}{{\bf f}^3}\otimes\frac{\bf f}{{\bf
r+f}^2}=\frac{{\bf rff+frf+f}^4}{{\bf rfr+frr+f}^4}
$$
Now observe that $\lim^i \frac{{\bf f}^2}{{\bf f}^3}\otimes
G_{ab}=\lim^i F_{ab}^{\otimes 2}\otimes G_{ab}=0,\ i\geq 0$. Hence
$$
{\sf Tor}(G_{ab}^{\otimes 2}, G_{ab})=\lim \frac{{\bf
rff+frf+f}^4}{{\bf rfr+frr+f}^4}.
$$
Now consider the exact sequence
$$
0\to \frac{{\bf rff+frf+f}^4}{{\bf rfr+frr+f}^4}\to \frac{{\bf
f}^3}{{\bf rfr+frr+f}^4}\to \frac{{\bf f}^3}{{\bf rff+frf+f}^4}\to
0
$$
By Lemma \ref{lemma49},
$$
 \frac{{\bf f}^3}{{\bf rff+frf+f}^4}=\frac{{\bf f}^2}{{\bf
 rf+fr+f}^3}\otimes \frac{\bf f}{{\bf f}^2}=G_{ab}^{\otimes
 2}\otimes F_{ab},
$$
hence
$$
\lim^i\frac{{\bf f}^3}{{\bf rff+frf+f}^4}=0,\ i\geq 0.
$$
Therefore,
$$
{\sf Tor}(G_{ab}^{\otimes 2}, G_{ab})=\lim\frac{{\bf
rff+frf+f}^4}{{\bf rfr+frr+f}^4}=\lim^1({\bf rff+frf+f}^4)
$$
and the needed isomorphism (\ref{theiso}) is proved.

Some other examples of description of limits can be found in
\cite{Ivanov-Mikhailov} and \cite{MP}. Here we give some examples
(see \cite{Ivanov-Mikhailov}, \cite{MP})\footnote{We assume that
the commutator brackets are left-normalized. For example,
$[R,R,F]:=[[R,R],F]$.}:
\begin{align*}
& \lim \frac{[F,F]}{[R,R][F,F,F]}=L_1S^2(G_{ab}),\\
& \lim \frac{[F,F,F]}{[R,R,R][F,F,F,F]}=L_2{\mathfrak
L}_s^3(G_{ab}),\\
& \lim \frac{[F,F,F]}{[R,R,F][F,F,F,F]}=L_1S^3(G_{ab}),\\
& \lim \frac{[F,F,F]}{[F,F,R][F,F,F,F]}=0,\\
& \lim \frac{[R,R]}{[R,R,F]}=H_2(G, S^2(I(G))),
\end{align*}
where $S^2,S^3, {\mathfrak L}_s^3$ are functors of symmetric
square, cube and super-Lie-cube respectively. If $G$ is 2-torsion
free then, as it is shown in \cite{Ivanov-Mikhailov},
\begin{align*} &
\lim \frac{[R,R]}{[R,R,F]}=H_4(G;\mathbb Z/2),\\
& \lim^1 \frac{[R,R]}{[R,R,F]}=H_3(G;\mathbb Z/2).
\end{align*}
\section{Intersection games}
Let ${\bf a}_1,\dots, {\bf a}_n,\ n\geq 2$ be ideals in $\mathbb
Z[F]$ which can be presented as monomials in ${\bf f}, {\bf r}$.
There is a natural problem: define a minimal set ${\bf b}_1,\dots,
{\bf
b}_m,\ m\geq 1$ of monomials in ${\bf f}, {\bf r}$ such that:\\ \\
1) All of them lie in intersection ${\bf a}_1\cap \dots \cap {\bf
a}_n$;\\ \\2) Any monomial ${\bf b}$ in ${\bf f}, {\bf r}$ which
lies in ${\bf a}_1\cap \dots \cap {\bf
a}_n$ also lies in some ${\bf b}_i$ as an ideal.\\

After constructing the set of monomials ${\bf b}_1,\dots, {\bf
b}_m,$ if $m\geq 2$, one can repeat this process. In this way we
obtain what we will call {\it generations} of collections of
ideals.

Here is our main example. Take two ideals ${\bf r}, {\bf f}^2$.
One can easily check that any monomial written in terms of ${\bf
f}$ and ${\bf r}$ which lies in ${\bf r}\cap {\bf f}^2$, lies
either in ${\bf rf}$ or in ${\bf fr}$ (or in both). Now we take
monomials from the intersection ${\bf rf}\cap {\bf fr}$ an see
that the minimal set of monomials which covers all of them is
${\bf r}^2, {\bf frf}$. Repeating this process, we obtain the
following {\it generations}: \begin{equation}\label{examplefr}
({\bf r},{\bf f}^2)\rightarrow ({\bf rf}, {\bf fr})\rightarrow
({\bf r}^2, {\bf frf})\rightarrow ({\bf r}^2{\bf f}, {\bf f}{\bf
r}^2)\rightarrow({\bf r}^3,{\bf fr}^2{\bf f})\rightarrow\dots
\end{equation}
In this example, we obtain the following isomorphism which follows
from Gruenberg resolution:
$$
\frac{\text{intersection of ideals in}\ i\text{th
generation}}{\text{sum of ideals in}\ (i+1)\text{st
generation}}=H_{i+1}(G).
$$

We collect some examples of origins and generators in the
following table:

\vspace{.5cm}
\begin{tabular}{|r||c|c|c|c|c|}
\hline {\sf origin} & {\bf r}, {\bf ff} & {\bf rr}, {\bf fff} & {\bf r}, {\bf fff}\\
\hline {\sf 2 generation} & {\bf rf}, {\bf fr} & {\bf rrf}, {\bf rfr}, {\bf frr} & {\bf rff}, {\bf frf}, {\bf ffr}\\
\hline {\sf 3 generation} & {\bf rr}, {\bf frf} & {\bf rrr}, {\bf frfrf} & {\bf rrr}, {\bf frrf}, {\bf rfrf}, {\bf frfr}, {\bf ffrff}\\
\hline {\sf 4 generation} & {\bf rrf}, {\bf frr} & {\bf frrrf}, {\bf frfrr}, {\bf rrfrf} & {\bf frrrf}, {\bf frfrfr}, {\bf rfrfrf}\\
\hline {\sf 5 generation} & {\bf rrr}, {\bf frrf} & {\bf rrrrr}, {\bf frfrfrf} & {\bf rrrrrr}, {\bf frfrfrf}, {\bf rfrrrfr}, {\bf frrrrfr}, {\bf rfrrrrf}\\
\hline
\end{tabular}
\vspace{.5cm}

In the above example (\ref{examplefr}), we see that
\begin{equation}\label{shift} \lim^{j+1}\ (\text{sum of ideals in the}\
(i+1)\text{st generation})=\lim^j\ (\text{sum of ideals in the}\
i\text{th generation})
\end{equation}
for $j\geq 1$. The reason for that connection is simple. All
generations consists of two ideals, say $({\bf a}_i, {\bf b}_i),\
i=1,2,\dots$ and there are natural short exact sequences:
$$
0\to \frac{{\bf a}_i\cap {\bf b}_i}{{\bf a}_{i+1}+{\bf
b}_{i+1}}\to \frac{{\bf f}}{{\bf a}_{i+1}+{\bf b}_{i+1}}\to
\frac{\bf f}{{\bf a}_i}\oplus \frac{\bf f}{{\bf b}_i}\to \frac{\bf
f}{{\bf a}_i+{\bf b}_i}\to 0
$$
Now (\ref{shift}) follows from the following properties:
$$
\lim^j\frac{{\bf a}_i\cap {\bf b}_i}{{\bf a}_{i+1}+{\bf
b}_{i+1}}=0,\ j>0,\ \lim^j\left(\frac{\bf f}{{\bf a}_i}\oplus
\frac{\bf f}{{\bf b}_i}\right)=0,\ j\geq 0.
$$
It would be interesting to explore general connections between
derived limits and above intersection game.

\newpage
\section{Table}
In the next table we summarize the computations of derived limits
for certain ${\bf fr}$-sentences. In this table we will denote the
augmentation ideal of $G$ by ${\sf g}$. \vspace{.5cm}

\begin{tabular}{|r||c|c|c|c|c|c|c|c}
\hline  {\sf fr-code} & ${\sf lim}^1$&${\sf lim}^2$&${\sf
lim^3}$&${\sf lim^4}$
\\ \hline {\sf f} & 0 & 0 & 0 & 0
\\ \hline {\sf r} & {\sf g} & 0 & 0 & 0
\\ \hline {\sf rr} & 0 & ${\sf g}\otimes {\sf g}$ & 0 & 0
\\ \hline {\sf rrr} & 0 & 0 & ${\sf g}\otimes {\sf g}\otimes {\sf g}$ & 0
\\ \hline {\sf rrrr} & 0 & 0 & 0 & ${\sf g}\otimes {\sf g}\otimes
{\sf g}\otimes {\sf g}$
\\ \hline {\sf fr+rf} & ${\sf g}\otimes_{\mathbb Z[G]}{\sf g}$ & 0 & 0 & 0
\\ \hline {\sf ffr+frf+rff} & ${\sf g}\otimes_{\mathbb Z[G]}{\sf g}\otimes_{\mathbb Z[G]}{\sf g}$ & 0 & 0 & 0
\\ \hline {\sf r+ff} & $G_{ab}$ & 0 & 0 & 0
\\ \hline {\sf r+fff} & ${\sf g}/{\sf g}^3$ & 0 & 0 & 0
\\ \hline {\sf rf+ffr} & ${\sf g}^2\otimes_{\mathbb Z[G]}{\sf g}$ & 0 & 0 & 0
\\ \hline {\sf rf+fffr} & ${\sf g}^3\otimes_{\mathbb Z[G]}{\sf g}$ & 0 & 0 & 0
\\ \hline {\sf rfr+frr+ffff} & ${\sf Tor}(G_{ab}\otimes G_{ab},G_{ab})$ & 0 & 0 & 0
\\ \hline {\sf fr+rf+fff} & $G_{ab}\otimes G_{ab}$ & 0 & 0 & 0
\\ \hline {\sf rff+frf+rff+ffff} & $G_{ab}\otimes G_{ab}\otimes G_{ab}$ & 0 & 0 & 0
\\ \hline {\sf rr+fff} & ${\sf Tor}(G_{ab},G_{ab})$ & $G_{ab}\otimes G_{ab}$ & 0 & 0
\\ \hline {\sf rrr+ffff} & $L_2\otimes^3(G_{ab})$ & $L_1\otimes^3(G_{ab})$ & $G_{ab}\otimes G_{ab}\otimes G_{ab}$ & 0
\\ \hline {\sf rrrr+fffff} & $L_3\otimes^4(G_{ab})$ & $L_2\otimes^4(G_{ab})$ & $L_1\otimes^4(G_{ab})$ &
$G_{ab}\otimes G_{ab}\otimes G_{ab}\otimes G_{ab}$
\\ \hline {\sf rr+frf} & $H_3(G)$ & ${\sf g}\otimes_{\mathbb Z[G]}{\sf g}$ & 0 & 0
\\ \hline {\sf rrf+frr} & $H_4(G)$ & $H_3(G)$ & ${\sf g}\otimes_{\mathbb Z[G]}{\sf g}$ &
0
\\ \hline {\sf rrr+frrf} & $H_5(G)$ & $H_4(G)$ & $H_3(G)$ &
${\sf g}\otimes_{\mathbb Z[G]}{\sf g}$
\\ \hline {\sf rrrf+frrr} & $H_6(G)$ & $H_5(G)$ & $H_4(G)$ &
$H_3(G)$
\\ \hline {\sf rf+ffr+ffff} & ${\sf g}^2/{\sf g}^3\otimes G_{ab}$ & 0 & 0 & 0
\\ \hline {\sf rfff+rfr+rrf} & 0 & ${\sf g}\otimes G_{ab}\otimes G_{ab}$ & 0 & 0
\\ \hline {\sf rrfff+rrfr+rrrf} & 0 & 0 & ${\sf g}\otimes {\sf g}\otimes G_{ab}\otimes G_{ab}$ & 0
\\ \hline
\end{tabular}

\end{document}